\documentclass[11pt]{article}
\usepackage{graphicx}
\usepackage{multirow}
\usepackage[margin=1in]{geometry}
\usepackage{amsmath,amssymb,amsfonts}
\usepackage{amsthm}
\usepackage{mathrsfs}
\usepackage[title]{appendix}
\usepackage[dvipsnames]{xcolor}
\usepackage{textcomp}
\usepackage{manyfoot}
\usepackage{booktabs}
\usepackage{algorithm,algorithmic}
\usepackage{listings}
\usepackage[square,numbers]{natbib}
\usepackage{enumerate}
\usepackage{bbm}
\usepackage{mathtools}
\usepackage{tikz}
\usepackage{circledsteps}
\usepackage{hyperref}

\newenvironment{keywords}{%
  \subsection*{Keywords}
	}{}
\usepackage{cleveref}
\crefname{assumption}{Assumption}{Assumptions}
\Crefname{assumption}{Assumption}{Assumptions}
\crefname{condition}{Condition}{Conditions}
\Crefname{condition}{Condition}{Conditions}
\newtheorem{theorem}{Theorem}
\newtheorem{proposition}[theorem]{Proposition}

\newtheorem{remark}[theorem]{Remark}
\newtheorem{definition}[theorem]{Definition}
\newtheorem{assumption}{Assumption}
\newtheorem{corollary}[theorem]{Corollary}
\newtheorem{lemma}[theorem]{Lemma}
\usepackage[normalem]{ulem}

\usetikzlibrary{patterns}

\newcommand{\R}{\mathbb{R}}

\newcommand{\sobj}{f}
\newcommand{\nobj}{\varphi}
\newcommand{\ared}{A}
\newcommand{\cred}{C}
\newcommand{\pred}{P}
\newcommand{\aredr}{{\rm ared}}
\newcommand{\credr}{{\rm cred}}
\newcommand{\predr}{{\rm pred}}
\newcommand{\kfcd}{\kappa_{\rm{fcd}}}
\newcommand{\kbmh}{\kappa_{\rm{bmh}}}
\newcommand{\kgrad}{\kappa_{\rm{grad}}}
\newcommand{\kef}{\kappa_{\rm{val}}}

\newcommand{\kobj}{\kappa_{\rm{obj}}}
\newcommand{\dom}{\mathrm{dom}\,}
\newcommand{\prox}[1]{\mathrm{prox}_{#1}}
\newcommand{\argmin}{\operatorname*{arg\,min}}
\newcommand{\dpsi}{\Psi_{k+1} - \Psi_k}

\newcommand{\one}{{\mathbbm{1}}}
\newcommand{\filt}{\mathcal{F}}

\title{ProxSTORM---A Stochastic Trust-Region Algorithm for Nonsmooth Optimization}
\author{
Robert J. Baraldi\footnotemark[1],\and
Aurya Javeed\footnotemark[1],\and
Drew P. Kouri\footnotemark[1], \and
Katya Scheinberg\footnotemark[2]
}

\ifpdf
\hypersetup{
	pdftitle={ProxSTORM},
	pdfauthor={Robert J. Baraldi, Aurya Javeed, Drew P. Kouri, and Katya Scheinberg}
}
\fi

\begin{document}

\maketitle
\renewcommand{\thefootnote}{\fnsymbol{footnote}}
\footnotetext[1]{
Optimization and Uncertainty Quantification, 
Sandia National Laboratories, PO Box 5800, 
Albuquerque, 87185-1320, NM, USA}
\footnotetext[2]{
H. Milton Stewart School of Industrial and Systems Engineering, 
Georgia Institute of Technology, 
755 Ferst Drive, Atlanta, 30332-0205, 
GA, USA
}

\begin{abstract}
We develop a stochastic trust-region algorithm for minimizing the sum of a Lipschitz-smooth but possibly nonconvex function and a convex but possibly nonsmooth function.
Such a problem class arises in many applications, including data science, operations research, and 
PDE-constrained optimization. 
This algorithm, which we call ProxSTORM, generalizes STORM \cite{storm,storm-rates}---a stochastic trust-region algorithm for the unconstrained optimization of smooth functions---and the inexact deterministic proximal trust-region algorithm in \cite{baraldi.2022}.
In the absence of a nonsmooth term, we recover the original STORM algorithm,  moreover, we improve and simplify certain aspects of STORM analysis, while maintaining 
 STORM martingale framework arguments to prove global convergence and an expected complexity bound.
We demonstrate ProxSTORM capabilities on neural network training and topology optimization under uncertainty.
\end{abstract}

\begin{keywords}
nonsmooth optimization, stochastic optimization, trust-region algorithm, proximal mapping, convex constraints
\end{keywords}

\section{Introduction}
We develop a new stochastic trust-region algorithm for the following class 
of composite optimization problems:
\begin{equation}\label{eq:p}
  \underset{x\in\R^d}{\text{minimize}}\;\; \left\{ \sobj(x) + \nobj(x) \right\},
\end{equation}
where $\nobj$ is convex but generally nonsmooth and $\sobj$ is smooth but typically nonconvex.
We require that $\sobj$ and $\nobj$ satisfy the following conditions.
\begin{assumption}\label{a:p}
\phantom{-}
\begin{enumerate}
\item
The function $\nobj:\R^d\to (-\infty,\infty]$ is proper, closed, and convex.
\item
The function $\sobj$ is $L$-smooth on $\dom\nobj := \{x: \nobj(x) <\infty\}$, i.e., $\sobj$ is differentiable on an open set $U\supseteq \dom{\nobj}$ and has $L$-Lipschitz continuous gradient. That is, there exists $L>0$ such that for all $x,y\in U$
$$
  \|\nabla \sobj(y) - \nabla \sobj(x)\| \le L\|y - x\|.
$$
\item The objective function $\sobj+\nobj$ is bounded below on $\dom\nobj$.
\end{enumerate}
\end{assumption}
Our algorithm exploits $\nobj$ whose proximal operator, defined as
\begin{equation}\label{eq:pm}
  \prox{r\nobj}(x) \coloneqq \argmin_{y\in\R^d} \left\{\nobj(y) + \frac{1}{2r}\|y - x\|^2\right\}
\end{equation}
can be evaluated exactly.
In general, \eqref{eq:pm} is a nontrivial optimization problem; an example is the proximal mapping of total variation $\|\nabla x\|$, even in finite dimensions. 
However, many common functions evaluated inexpensively or even analytically, such as the $L^1$-norm and  projections onto convex constraints. 
We refer the reader to \cite{parikh.2014} for an in-depth treatment. 
Our problem class is thus suitable 
for such $\nobj$ and the case that, e.g., 
$\sobj$ is an expected loss with respect to a data distribution.
For $\sobj$, we allow stochastic models and estimates that can be arbitrarily inaccurate, provided that:
($i$) these quantities are sufficiently accurate, sufficiently often; and
($ii$) the objective function $\sobj+\nobj$ does not increase too much when estimates of $\sobj$ are inaccurate.
These assumptions are both practical and widespread in the literature;  c.f. \cite{storm,storm-rates} and citations therein.

It is in this context that we develop a stochastic algorithm 
for \eqref{eq:p}, called ProxSTORM. 
ProxSTORM combines STORM, 
a trust-region framework for \underline{st}ochastic \underline{o}ptimization with \underline{r}andom \underline{m}odels, 
\cite{storm,storm-rates} with the proximal trust-region framework in \cite{baraldi.2022}.
The latter framework imposes the same conditions 
as Assumption \ref{a:p}.1 and permits \emph{deterministic} inexactness concerning $\sobj$.
The inexactness criteria used in in \cite{baraldi.2022} 
stipulates that models and estimates of $f$ always 
satisfy error bounds decreasing to zero as the algorithm progresses.
In many practical applications, however, inexactness is \emph{stochastic}, resulting from sampled approximations of expectations or sporadic failures of subroutines used to evaluate $\sobj$ that arise from sophisticated simulations \cite{storm}.
To accommodate stochastic inexactness in the context of \eqref{eq:p}, 
we generalize STORM to accommodate nonsmooth $\nobj$. 
In keeping with \cite{baraldi.2022}, we use the norm of the $\sobj+\nobj$ \emph{proximal gradient} (PG) step as our stationarity measure; that is, ProxSTORM terminates 
when the norm of
\begin{equation}\label{eq:gg}
  h(x) \coloneqq \frac{x - \prox{r\nobj}\left(x - r\nabla f(x)\right)}{r},
\end{equation}
is less than a user-supplied tolerance. 
Here, $r$ is a constant, user-specified parameter,
that does not affect the convergence analysis.
A stationary point $x^*$ of $\sobj+\nobj$ 
satisfies $h(x^*) = 0$ with $h$ reducing to $\nabla \sobj$ when $\nobj\equiv 0$.
Our contribution is a nonsmooth, stochastic algorithm 
that accommodates $\nobj$ 
while recovering theory no worse than the theory for STORM when $\nobj\equiv 0$.
In adjusting the analysis from STORM, we primarily
utilize both the assumption that $\sobj+\nobj$ does not increase 
too much when estimates of $sobj$ are inaccurate, 
and the nonexpansivity of the proximal mapping from 
\cite[Proposition 12.27]{bauschke.2017}:
\begin{equation}\label{eq:ne}
  \|\prox{r\nobj}(y) - \prox{r\nobj}(x)\|\le \|y - x\|
  \quad\text{for all}\quad
  x, y\in\R^d.
\end{equation}
ProxSTORM focuses in particular on \eqref{eq:p} with
stochastically approximated $\sobj$ and deterministic $\nobj$.
For example, $\sobj$ can be a mean over some uncertain variable, 
whereas $\nobj$ can be an $L^1$-norm or the indicator onto convex constraints. 
Several recent papers propose optimization algorithms for stochastic objective function inexactness and deterministic constraints.
These methods, however, do not involve nonsmooth functions and are based on constrained optimization techniques, such as sequential quadratic programming \cite{BeraCurtRobiZhou21,CurtJianWang24b} and interior point methods \cite{CurtKungRobiWang23,CurtJianWang24b}.
Our ProxSTORM algorithm only addresses convex constraints but our approach and the assumptions of our analysis are significantly simpler and easier to enforce than those for general stochastic constrained optimization.
Aside from STORM \cite{storm}, other stochastic trust-region methods include \cite{cao2024first,Grattonetal2017}.
We find, however, that the framework of STORM is the most suitable for our purposes since it is the most general and also yields expected complexity bounds.

In addition to the novel ProxSTORM algorithm and its analysis this paper also contributes to the understanding of the original STORM algorithm. Via an improved analysis (which applies both to ProxSTORM and STORM) we make it evident that function estimates are utilized only via their differences - that is what we need to estimate if a function {\em change} rather than the function itself. This leads to an important revelation that STORM (and ProxSTORM) have sample complexity which is comparable with any other algorithm when applied to problems in many setting including machine learning and empirical risk minimization. 

The paper is organized as follows.
In Section~\ref{sec:ps}, we present the ProxSTORM algorithm.
In Section~\ref{sec:st}, we formalize the notation we use to describe the stochasticity of the algorithm.
Then, in Section~\ref{sec:an}, we analyze ProxSTORM.
We combine Assumption~\ref{a:p} with stochasticity assumptions (Section~\ref{sec:as}) to derive an expected decrease result (Section~\ref{sec:ed}).
From the expected decrease result, we establish global first-order convergence (Section~\ref{sec:gc}) and an expected complexity bound (Section~\ref{sec:co}).
In Section~\ref{sec:oracles} we discuss how the assumptions on function and gradient estimates used in the analysis can be satisfied in applications and the resulting total oracle 
complexity.
In Section~\ref{sec:ex}, we apply ProxSTORM to $\ell^1$-regularized neural network training and a topology optimization problem.
In Section~\ref{sec:cn}, we conclude.

\section{ProxSTORM Algorithm}\label{sec:ps}

Given a (user-specified) starting point $x_0$, ProxSTORM generates a sequence of iterates $x_k$, $k = 1, 2,\ldots$, where $x_{k+1}=x_k+s_k$ on successful iterations and $x_{k+1}=x_k$ on unsuccessful iterations.
Each iteration consists of computing a trial step $s_k$ that is an approximate solution to the trust-region subproblem
\begin{equation*}
  \underset{||s||\le\delta_k}{\text{minimize}}\;\; \{m_k(x_k+s) + \nobj(x_k+s)\}.
\end{equation*}
Here, $m_k$ is a random model of the smooth part of the objective function, $f$, and $\delta_k$ is the radius of the trust region around $x_k$, wherein $m_k$ is presumed valid.
We take
$$
  h_k\coloneqq \frac{x_k - \prox{r\nobj}(x_k-r\nabla m_k(x_k))}{r},
$$
which is the proximal gradient of $m_k+\nobj$ at $x_k$, and
\begin{equation}\label{eq:b}
  b_k \coloneqq 1+\max_{\|x-x_k\| \le \delta_k}\max_{w\neq 0} \frac{|w^\top \nabla^2 m_k(x)w|}{\|w\|^2}.
\end{equation}
We require that approximate subproblem solutions 
$s_k$ satisfy trust-region feasibility
and a generalized notion of fraction of Cauchy decrease \cite{baraldi.2022,conn.2000}:
\begin{enumerate}
\item[(S1)]
$s_k$ belongs to the trust region (i.e., the ball of radius $\delta_k$ around $x_k$); 
and
\item[(S2)]
the predicted reduction, $\predr_k$, satisfies the fraction of Cauchy decrease condition, i.e., there exists $\kfcd > 0$ independent of $k$ such that
\end{enumerate}\vspace{-1em}
\begin{align*}
  \predr_k &\coloneqq  m_k(x_k)+\nobj(x_k) - m_k(x_k+s_k)-\nobj(x_k+s_k)
  \ge \kfcd\|h_k\|\min\left\{\frac{\|h_k\|}{b_k}, \delta_k\right\}.
\end{align*}
Condition (S2) is analogous to the canonical fraction of Cauchy decrease condition for smooth problems \cite[AA.1b]{conn.2000}, and recovers the smooth
condition when $\nobj \equiv 0$. 
Given $s_k$ satisfying (S1) and (S2), ProxSTORM either accepts or rejects $s_k$ (corresponding to a successful or unsuccessful iteration, respectively) by comparing $\predr_k$ with an approximation of the actual reduction,
$$
  \aredr_k \coloneqq \sobj(x_k) + \nobj(x_k) - \sobj(x_k+s_k) - \nobj(x_k+s_k).
$$
We call this approximation the \emph{computed reduction} and denote it as $\credr_k$.
This approximation corresponds to only evaluating 
stochastic estimates function values.
A standard step acceptance condition for inexact trust-region algorithms is
$$
  \frac{\credr_k}{\predr_k}\ge \eta_1.
$$
Similar to STORM, ProxSTORM uses this step acceptance condition and requires that the trust-region radius be at least a constant factor smaller than the stationarity measure.
Specifically, to accept a step, ProxSTORM stipulates that for fixed (user-specified) parameters $\eta_1\in (0,1)$ and $\eta_2 > 0$,
\begin{equation}\label{eq:acceptance1}
  \frac{\credr_k}{\predr_k}\ge \eta_1
  \qquad\text{and}\qquad
  \|h_k\|\ge \eta_2\delta_k.
\end{equation}
When \eqref{eq:acceptance1} is satisfied, the iterate and trust-region radius is updated as
\begin{equation}\label{eq:successful}
  x_{k+1}=x_k+s_k \qquad\text{and}\qquad \delta_{k+1}=\min\{\gamma\delta_k,\delta_{\text{max}}\}.
\end{equation}
Here, $\delta_{\text{max}} := \gamma^\ell\delta_0$, where $\gamma > 1$, $\ell\in\mathbb{N}$, and $\delta_0 > 0$ are fixed (user-specified) parameters.
When \eqref{eq:acceptance1} is not satisfied,
\begin{equation}\label{eq:unsuccessful}
  x_{k+1}= x_k
  \qquad\text{and}\qquad
  \delta_{k+1}= \gamma^{-1}\delta_k.\hspace{3.3em}
\end{equation}
We state ProxSTORM as Algorithm~\ref{alg:ps}.
Note that Algorithm~\ref{alg:ps} checks $\|h_k\|<\eta_2\delta_k$ before computing $s_k$ since $s_k$ is unused when that inequality holds, i.e., $s_k$ will be rejected.
\begin{algorithm}[H]
\begin{algorithmic}
\REQUIRE the following non-random quantities:
\vspace{-0.5em}
\begin{itemize}
\item parameters $\eta_1\in (0,1)$, $\eta_2 > 0$, $\gamma > 1$, $\ell\in\mathbb{N}$,
\item
initial iterate $x_0\gets x \in \dom{\nobj}$,
\item
initial trust-region radius $\delta_0 > 0$.
\end{itemize}
\vspace{-0.5em}
\FOR{$k=0,1,\ldots$}
\STATE \textbf{Select Model:} Choose a model $m_k$.
\IF{$\|h_k\|<\eta_2\delta_k$}
  \STATE \emph{Unsuccessful Step:} Set $x_{k+1}$ and $\delta_{k+1}$ according to \eqref{eq:unsuccessful}.
  \STATE $\credr_k = \aredr_k = 0$
  \STATE \textbf{continue}
\ENDIF
  \STATE \textbf{Compute Step:} Compute $s_k$ that satisfies (S1) and (S2).
  \STATE \textbf{Compute Reduction:} Compute $\credr_k$.
  \IF{$\frac{\credr_k}{\predr_k}\ge\eta_1$}
    \STATE \emph{Successful Step:} Set $x_{k+1}$ and $\delta_{k+1}$ according to \eqref{eq:successful}.
  \ELSE
    \STATE \emph{Unsuccessful Step:} Set $x_{k+1}$ and $\delta_{k+1}$ according to \eqref{eq:unsuccessful}.
  \ENDIF
\ENDFOR
\end{algorithmic}
\caption{ProxSTORM---a stochastic trust-region method for minimizing $\sobj+\nobj$.}
\label{alg:ps}
\end{algorithm}

The primary difference between Algorithm~\ref{alg:ps} and STORM is the nonsmooth term $\nobj$, which appears in the trust region subproblem and in the $\aredr_k$ term that $\credr_k$ approximates.
Due to $\nobj$, the $h_k$ in Algorithm~\ref{alg:ps} is not the gradient of $m_k$ as in STORM, but the weighted proximal gradient step of $m_k+\nobj$.
Our algorithm is also similar to the deterministic framework \cite{baraldi.2022}, but we have the additional step acceptance criteria $\|h_k\|\ge \eta_2\delta_k$ and a different trust-region update;
see \cite[Algorithm 1]{baraldi.2022} for a full description. 
Our choice of trust-region update is matter of convenience, and while a more flexible update is possible, it complicates the analysis. 

\paragraph{Computing the Trial Step $s_k$}
There are several algorithms for generating trial steps that satisfy (S1) and (S2).
See \cite{baraldi.2025}.
Each of the methods in \cite{baraldi.2025} employs a quadratic model
$$
  m_k(s+x_k) = \tfrac{1}{2}s^\top Q_k s + g_k^\top s,
$$
where $g_k=\nabla m_k(x_k) \approx \nabla f(x_k)$ and $Q_k$ is a symmetric $n\times n$ matrix that characterizes the curvature of $f$ around $x_k$; for instance, $Q_k$ can be the Hessian of $f$ when it exists or an approximation thereof.
The methods in \cite{baraldi.2025} first determine a Cauchy point.
There are two Cauchy point definitions in \cite{baraldi.2025}, but we only discuss the original Cauchy point introduced in \cite{baraldi.2022}.
Let $p_k(r)$ denote the Cauchy arc
$$
  p_k(r) \coloneqq \prox{r\nobj}(x_k-r\nabla m_k(x_k)).
$$
To determine a Cauchy point, one can employ the bi-directional proximal search presented as \cite[Algorithm~2]{baraldi.2022}.
The search produces a step length $r_k > 0$ that satisfies \cite[(25)]{baraldi.2022}, in which case
the Cauchy step defined by
$$
  s^c_k \coloneqq p_k(r_k) - x_k,
$$
satisfies (S1) and (S2).
The methods in \cite{baraldi.2025} improve upon the Cauchy step $s^c_k$ with iterative procedures such as spectral proximal gradient, proximal nonlinear conjugate gradient, or semismooth Newton.
These methods produce descending model values at each iteration while enforcing the trust-region constraint, so the computed trial step $s_k$ continues to satisfy (S1) and (S2).
For our numerical results, we use \cite[Algorithm~5]{baraldi.2022}---a spectral proximal gradient method.

\section{Stochastic Formalism}\label{sec:st}

In the preceding section, we introduced ProxSTORM in terms of a realization of the stochastic algorithm, which we interpret as a point $\omega$ in a probability space $(\Omega,\mathcal{F},\mathbb{P})$.
Each $\omega$ encapsulates \emph{all} of the randomness when running the algorithm.
The realizations are elements of the set $\Omega$, $\mathcal{F}\subseteq 2^\Omega$ is a $\sigma$-algebra, and $\mathbb{P}:\mathcal{F}\to[0,1]$ is a probability measure.
These concepts are critical to proving the global convergence
(Section~\ref{sec:an}) and establishing complexity results (Sections~\ref{sec:an} and~\ref{sec:oracles})
of ProxSTORM.
Randomness enters the algorithm through the models $m_k$ of $\sobj$ and through the computed reductions $\credr_k$.
These quantities are in fact \emph{random variables}, i.e., measurable functions on $(\Omega,\mathcal{F},\mathbb{P})$.
Entities in the algorithm that depend on them are thus random variables as well.
We generally use capital letters for random variables and lowercase text for their realizations.
Concretely, for a realization $\omega$ of ProxSTORM at iteration $k$, the model and the computed reduction are
\begin{equation}\label{eq:realizations}
  m_k = M_k(\cdot;\omega):\R^d\to\R \quad\text{and}\quad \credr_k = \cred_k(\omega)\in\R.
\end{equation}
These variables are the source of randomness for the entire algorithm. $M_k$ is realized first during iteration $k$ and $C_k$ is realized at a later stage of this iteration.  
Formally, ProxSTORM is defined as a \emph{stochastic process} adapted to a filtration
\begin{equation}\label{eq:filt}
  \mathcal{F}_0 \subseteq \mathcal{F}_{\frac{1}{2}} \subseteq \mathcal{F}_1 \subseteq \mathcal{F}_{\frac{3}{2}} \subseteq \cdots \subseteq \mathcal{F}.
\end{equation}
Each $\mathcal{F}_k$, $k\in\mathbb{N}$, is the smallest $\sigma$-algebra generated by $\{M_\ell,\cred_\ell\}_{\ell\le k}$, which are the sources of randomness through iteration $k$.
Informally, $\mathcal{F}_k$ is the \emph{information} created through iteration $k$ \cite{durrett5}.
Each half iterate, $\mathcal{F}_{k+\frac{1}{2}}$, is the smallest $\sigma$-algebra generated by $\mathcal{F}_k$ and $M_{k+1}$ together, i.e., the information created in iteration $k+1$, up to but not including $\cred_{k+1}$.
The quantities computed by the algorithm but that do not produce additional randomness are:
\begin{multline*}
x_k=X_k(\omega),
\quad\delta_k = \Delta_k(\omega),
\quad h_k = H_k(\omega),\\
\quad s_k = S_k(\omega),
\quad\predr_k=\pred_k(\omega),
\quad\text{and}
\quad\aredr_k=\ared_k(\omega).
\end{multline*}
The random variables $X_{k+1}$ and $\Delta_{k+1}$ are adapted to $\mathcal{F}_k$; the information from iteration $k$ is enough to specify $X_{k+1}$ and $\Delta_{k+1}$.
Similarly, $H_k$, $S_k$, $\pred_k$ and $\ared_k$ are adapted to $\mathcal{F}_{k-\frac{1}{2}}$.

As is standard in probability, we take an \emph{event} to be a measurable subset of $\Omega$.
For instance, given a (user-specified) constant $\kgrad > 0$, we define the event
\begin{equation*}
  \mathcal{I}_k = \left\{\omega\in\Omega : \|\nabla M_k(X_k(\omega);\omega) - \nabla\sobj(X_k(\omega)) \|\le \kgrad \Delta_k(\omega) \right\},
\end{equation*}
which will be our notion of a sufficiently accurate model for iteration $k$ of Algorithm~\ref{alg:ps}.
We follow the convention of suppressing the dependence on $\omega$ when prescribing events, writing, e.g., simply
\begin{equation}\label{eq:ik}
  \mathcal{I}_k := \left\{\|\nabla M_k(X_k) - \nabla\sobj(X_k) \|\le \kgrad \Delta_k \right\}.
\end{equation}
Our notion of a sufficiently accurate computed reduction for iteration $k$ of Algorithm~\ref{alg:ps} is
\begin{equation}\label{eq:jk}
  \mathcal{J}_k := \left\{\left|\ared_k - \cred_k \right|\le \eta\pred_k\right\},
\end{equation}
where $\eta\in \left(1,\min\{\eta_1,1-\eta_1\}\right)$.
We will also use
\begin{equation}\label{eq:ak}
  \mathcal{S}_k := \{\text{Algorithm~\ref{alg:ps} accepts the trial step } S_k \},
\end{equation}
and other events that we define later.

We also adopt the following conventions: ($i$) an unqualified condition involving random variables is a probability one event and ($ii$) an overline represents the complement of an event, e.g., 
$$\overline{\mathcal{I}_k} := \Omega \backslash \mathcal{I}_k.$$

\section{Global Convergence and Complexity Analysis}\label{sec:an}

In Section~\ref{sec:gc}, we establish that ProxSTORM converges globally, and in Section~\ref{sec:co}, we establish its the expected complexity.
In Section \ref{sec:as}, we present assumptions about $M_k$, $\cred_k$, and $\ared_k$ on which these results depend.
The most involved step is an expected decrease result that we state and prove in Section~\ref{sec:ed}.

\subsection{Assumptions}\label{sec:as}

We begin with assumptions about models of $\sobj$, the smooth part of the objective function.
\begin{assumption}[Bounded Model Curvature]\label{a:m1}
  The trust-region subproblem models, $\{M_k\}_{k=0}^\infty$, have uniformly bounded curvature, i.e., there exists a number $\kbmh > 0$ independent of $k$ and $\omega$ so that
\begin{equation*}\label{eq:mk}
  1 + b_k\le\kbmh
\end{equation*}
for every $\omega\in\Omega$, where $b_k$ is defined in \eqref{eq:b}.
\end{assumption}
\begin{assumption}[Model Accuracy]\label{a:m2}
  For all $k\in\mathbb{N}$, the trust-region subproblem model, $M_k$, is sufficiently accurate sufficiently often, meaning
for a number $0 < \alpha < 1$ sufficiently close to one that we set later,
\begin{equation*}
  \mathbb{P}(\mathcal{I}_k|\mathcal{F}_{k-1}) \ge\alpha.
\end{equation*}
\end{assumption}
Recall that $\mathcal{I}_k$ is the event \eqref{eq:ik}, which is sufficient for the difference between the actual and predicted reduction, $A_k - P_k$, to be small.
We isolate this fact as our first lemma.
\begin{lemma}[Accurate Model Gradient $\Rightarrow$ Accurate Predicted Reduction]\label{prop:kef}
  Suppose Assumptions~\ref{a:p} and~\ref{a:m1} hold.
Then on $\mathcal{I}_k$, 
\begin{equation*}
  \aredr_k - \predr_k \le 2\kef\delta_k^2
\end{equation*}
where
\begin{equation}\label{eq:kef}
  \kef \coloneqq \tfrac{1}{4}(L + \kbmh + 2\kgrad).
\end{equation}
\end{lemma}

\begin{proof}
Since $\predr_k$ uses exact evaluations of the nonsmooth part of the objective function, $\nobj$, the difference $\aredr_k - \predr_k$ depends only on the reductions of $\sobj$ and $m_k$.
We apply the descent lemma \cite[Lemma 2.64]{bauschke.2017} to the difference between these smooth functions:
\begin{multline*}
 \left(\sobj(x_k+s_k) - m_k(x_k+s_k)\right) - \left(\sobj(x_k) - m_k(x_k)\right) \\
  \le \left( \nabla\sobj(x_k) - \nabla m_k(x_k)\right)^\top s_k + \frac{L+\kbmh}{2}\left\|s_k\right\|^2.
\end{multline*}
Here, $L+\kbmh$ is an upper bound for the Lipschitz constant for the gradient of $\sobj-m_k$.
Taking the norm of the right-hand side of this expression and applying the Cauchy-Schwarz inequality establishes the lemma since
$$\|\nabla m_k(x_k) - \nabla f(x_k)\| \le \kgrad\delta_k$$
on $\mathcal{I}_k$.
\end{proof}

Assumption~\ref{a:m1} is that the Hessians of models are uniformly bounded always, whereas Assumption~\ref{a:m2} is that $M_k$ is a probabilistic oracle.
In particular, $M_k$ must be sufficiently accurate with sufficiently high probability but otherwise can be arbitrarily inaccurate.
We formulate the computed reduction $C_k$ in a similar manner.

\begin{assumption}[Computed Reduction Accuracy]\label{a:c}
  For all $k\in\mathbb{N}$, the computed reduction, $\cred_k$, is sufficiently accurate sufficiently often, meaning
for a number $0 < \beta < 1$ sufficiently close to one that we set later,
\begin{equation*}
  \mathbb{P}(\mathcal{J}_k |\filt_{k-\frac{1}{2}})\ge\beta.
\end{equation*}
\end{assumption}

Recall that $\mathcal{J}_k$ is the event \eqref{eq:jk}.
Assumption~\ref{a:c} streamlines its STORM counterpart in a manner analogous to Assumption~\ref{a:m2}; that is, for ProxSTORM (and for STORM), it suffices to consider only the difference between $\cred_k$ and $\ared_k$ instead of estimating the objective function accurately at both $X_k$ and $X_k+S_k$.
A second aspect of Assumption~\ref{a:c} that can be applied to STORM is the quantification of accuracy in terms of the predicted reduction $\pred_k$ instead of $\Delta_k^2$.
Note that Assumptions~\ref{a:m2} and~\ref{a:c} of ProxSTORM are invariant under constant shifts of $M_k$ and of objective function estimates.
Our use of $P_k$ is consistent with conventional objective function inexactness conditions for deterministic trust-region algorithms \cite{conn.2000,carter.1989}.
In the stochastic setting of ProxSTORM, the use of $\pred_k$ has two advantages.
First, it decouples the definition of $\mathcal{J}_k$ from the $\eta_2$ parameter in the latter step acceptance condition in \eqref{eq:acceptance1}.
Second it can be less restrictive since the fraction of Cauchy decrease condition (S2) ensures
$$\predr_k \ge \kfcd\|h_k\|\min\left\{\frac{\|h_k\|}{\kbmh},\delta_k\right\}.$$
which is at least $O(\delta_k^2)$ when $\|h_k\|\ge\eta_2\delta_k$ but could be larger.
For realizations $\omega$ where $\|h_k\|<\eta_2\delta_k$, $\credr_k = \aredr_k = 0$, so $\omega\in\mathcal{J}_k$.

Our final assumption pertains to the actual reduction, $\ared_k$, approximated by the computed reduction.
\begin{assumption}\label{a:new}
For all $k\in\mathbb{N}$, there exist numbers $c_1$, $c_2\ge 0$ independent of $k$ and $\omega$ such that
$$\mathbb{E}\left[-\ared_k\mathbbm{1}(\mathcal{S}_k\cap\overline{\mathcal{J}_k})\Big\vert \filt_{k-1/2}\right] \le c_1\Delta_k^2 + c_2\|h(X_k)\|\Delta_k.$$
\end{assumption}
This assumption controls the increase of the objective function in the event of a successful step and an inaccurate computed reduction.
One can show that Assumption~\ref{a:new} holds when $\nobj\equiv0$ and follows from the corresponding STORM assumption. 
\begin{proposition}[Applicability of Assumption~\ref{a:new}]\label{prop:new}
\ \newline\vspace{-1em}
\begin{itemize}
\item
  Suppose the objective function is smooth, i.e., $\nobj\equiv 0$.
Then Assumption~\ref{a:new} holds with $c_1 = \frac{L}{2}(1-\beta)$ and $c_2 = (1-\beta)$.
\item
Assumption~\ref{a:new} is implied by
\begin{equation}\label{eq:new1}
  \mathbb{E}\left[|\cred_k - \ared_k|\Big\vert\filt_{k-1/2}\right]\le c_3\Delta_k^2,
\end{equation}
with $c_1 = c_3$ and $c_2=0$.
\end{itemize}
\end{proposition}
\begin{proof}
When $\nobj\equiv 0$, $h(X_k) =\nabla f(X_k)$, so Assumption~\ref{a:new} follows from a pointwise application of the descent lemma to $-\aredr_k$ on $\mathcal{S}_k\cap \overline{\mathcal{J}_k}$.
In particular,
\begin{align*}
\mathbb{E}\left[-A_k \mathbbm{1}(\mathcal{S}_k\cap \overline{\mathcal{J}_k}) \Big\vert \mathcal{F}_{k-1/2}\right]
&\le
\mathbb{E}\left[\left(\|\nabla f(X_k)\|\Delta_k + \frac{L}{2}\Delta_k^2\right) \mathbbm{1}(\mathcal{S}_k\cap \overline{\mathcal{J}_k}) \Big\vert \mathcal{F}_{k-1/2}\right] \\
&\le
(1-\beta)\Big(\frac{L}{2}\Delta_k^2 + \|\nabla f(X_k)\|\Delta_k\Big).
\end{align*}
To show Assumption~\ref{a:new} is implied by \eqref{eq:new1}, observe that
\begin{equation*}
  \mathbb{E}\left[-\ared_k\mathbbm{1}(\mathcal{S}_k\cap\overline{\mathcal{J}_k})\Big\vert \mathcal{F}_{k-1/2}\right]
  \le \mathbb{E}\left[\left(|\cred_k-\ared_k| -\cred_k\right)\mathbbm{1}(\mathcal{S}_k\cap\overline{\mathcal{J}_k})\Big\vert \mathcal{F}_{k-1/2}\right].
\end{equation*}
On $\mathcal{S}_k$, $\cred_k$ is positive since $\cred_k\ge \eta_1\pred_k$ is a step acceptance condition and $\pred_k > 0$ due to the fraction of Cauchy decrease condition (S2).
By the monotonicity of expectations,
\begin{align*}
\mathbb{E}\left[-\ared_k\mathbbm{1}(\mathcal{S}_k\cap\overline{\mathcal{J}_k})\Big\vert \mathcal{F}_{k-1/2}\right] \le \mathbb{E}\left[|\cred_k-\ared_k|\Big\vert \mathcal{F}_{k-1/2}\right],
\end{align*}
so \eqref{eq:new1} implies Assumption~\ref{a:new} with $c_1$ and $c_2$ as claimed.
\end{proof}

We conclude this subsection by comparing Assumption~\ref{a:c} with the inexact objective function condition used by the deterministic algorithm \cite{baraldi.2022}, which first appeared in \cite{kouri.2014}.
The deterministic condition is that, for all $k\in\mathbb{N}$, there exists a number $\kobj>0$ independent of $k$ such that
\begin{equation}\label{eq:kouri}
  |\aredr_k - \credr_k|\le \kobj\left(\eta\min\left\{\predr_k,\theta_k\right\}\right)^\mu.
\end{equation}
Here, $\mu$ and $\theta_k$ are (user-specified) nonnegative numbers with
$$\mu > 1\quad\text{and}\quad \lim_{k\to\infty} \theta_k = 0,$$
and  $\eta\in(0,\min\{\eta_1,1-\eta_1\})$ as in Assumption~\ref{a:c}.

One can consider this condition in the ProxSTORM framework by relaxing the requirement that it hold deterministically, instead allowing it to fail with a sufficiently small probability.
This condition has more parameters than the condition  \eqref{eq:jk} whose  deterministic versions is  
\begin{equation}\label{eq:standard}
  |\aredr_k - \credr_k|\le \eta\predr_k.
\end{equation}

More parameters allows for better practical flexibility. 
For instance, $\kobj$ need not be specified, which enables the use of error indicators that may depend on uncomputable constants, to bound the left-hand side of \eqref{eq:kouri}.
When $k$ is sufficiently large, however, \eqref{eq:kouri} implies \eqref{eq:standard} since $\theta_k\to 0$ \cite[Lemma~A.1]{kouri.2014}.
We note that the definition of $\mathcal{J}_k$ in Assumption~\ref{a:c} could be modified to become a probabilistic version of \eqref{eq:kouri} without meaningfully altering the convergence analysis of ProxSTORM. However, similar to our decision to forgo a more flexible trust-region update, we refrain from doing so in the interest of clarity.

\paragraph{Further Discussion of Assumptions \ref{a:m1}--\ref{a:new}}
Assumption \ref{a:m1} the standard assumption in the trust region literature that bounds the norm of model Hessians.
Assumptions \ref{a:m2}-\ref{a:new} are assumptions on the computational oracles used by our trust region algorithms: specifically Assumptions \ref{a:m2} is the assumption on the inexact gradient that the method has access to and Assumptions \ref{a:c} and \ref{a:new} are the assumptions on the computed reductions.

As discussed above, Assumptions \ref{a:m2} and \ref{a:c} are similar to those in STORM \cite{storm}.
The STORM conditions are shown in \cite{storm} to be able to be satisfied in the standard stochastic optimization setting where one has access to unbiased estimates of the gradient and the function value.
Specifically, in the case where the estimates have a bounded variance, then averaging over an appropriately large sample and applying Chebyshev's inequality (or Bernstein, in the case of sub-Gaussian noise) gives a gradient estimate that satisfies assumption Assumption \ref{a:m2} with appropriate probability.
Such a procedure is not uncommon in the literature; we provide details in Section~\ref{sec:oracles}.
Similarly, one can obtain function estimates satisfying Assumption \ref{a:c}.
In STORM and in other related papers, Assumption \ref{a:c} was only analyzed under very general assumption on the function value noise---unbiased with bounded variance.
In this case, the number of samples to ensure Assumption \ref{a:c} is $\mathcal{O}(\Delta_k^{-4})$, which makes the method seem more computationally costly than standard stochastic methods.
In Section \ref{sec:oracles}, however, we will show how in many standard settings, such as empirical risk minimization, Assumption \ref{a:c} is ensured with only $\mathcal{O}(\Delta_k^{-2})$ samples.

We also note that Assumptions \ref{a:m2} and \ref{a:c} are meant to be much more general than the standard empirical risk minimization setting as they tolerate heavy tails and fully corrupted estimates, which is also discussed in \cite{storm}.
Assumptions \ref{a:m2} was first introduced in \cite{bandeira.2014} and was inspired by constructing sparse interpolation models in derivative free setting over random interpolation sets.
The follow-up works to STORM that used the same assumptions include \cite{Bellavia2023stochastic, Bellavia2022trust,Rinaldi2024stochastic}.
In all these works it is noted that as $\Delta_k$ goes to zero the sample size needs to be increased.
The dynamics of $\Delta_k$ in the STORM (and hence the setting in this paper) is analyzed in \cite{Jin25sample}.
A high probability of the lower bound on $\Delta_k$ is derived, which enables a total sample complexity bound.

We note that in practice it is often not necessary to increase the sample size when $\Delta_k$ gets smaller.
Specifically, while $\Delta_k$ can fluctuate, it tends to decrease as the algorithm gets closer to the solution.
In many applications, the level of noise decreases closer to the solution, thus reducing the need to increase the sample size.
We observe this in our computational results.

Finally, Assumption \ref{a:new}, can be simply seen as a relaxation of \eqref{eq:new1}, which is a bound on the variance of the function value estimates.
In the analysis of STORM, this bound is not needed because Assumption \ref{a:new} holds due to the second term in the right hand side.
In the ProxSTORM case, however, it is not guaranteed, so we combine the two bounds to ensure that the new results apply to STORM as well as ProxSTORM.

\subsection{Expected Decrease}\label{sec:ed}

The goal of this subsection is to establish the expected decrease of a function defined for each iteration of ProxSTORM in terms of the objective function and the trust-region radius.
We begin with some rudimentary lemmas about the behavior of the algorithm.
The first lemma concerns the event of a small trust-region radius.
In particular, let
\begin{equation}\label{eq:zeta}
  \mathcal{B}_k := \{\|h(X_k)\|\ge \zeta\Delta_k\},\quad  \zeta \ge \kgrad + \max\left\{\eta_2,\frac{4\kef}{(1-\eta_1 - \eta)\min\{\kfcd,1\}}\right\}.
\end{equation}
The first lemma is essentially \cite[Lemma~4.6]{storm}; it is the pointwise guarantee that on $\mathcal{I}_k\cap\mathcal{B}_k$, $\aredr_k$ is at least proportional to $\|h(x_k)\|\delta_k$, where we recall that $h(x_k)$ is the true (i.e., the objective function) proximal gradient step at $x_k$.
The event $\mathcal{I}_k$ is an accurate model, in which case $\aredr_k$ and $h(x_k)$ are well-approximated respectively by $\predr_k$ and $h_k$, the proximal gradient step of the model at $x_k$.
Hence, $\mathcal{I}_k\cap\mathcal{B}_k$ implies $\predr_k$ is at least proportional to $\|h(x_k)\|\delta_k$ due to the fraction of Cauchy decrease condition (S2).
The proof of the lemma is essentially an application of the descent lemma to $\aredr_k - \predr_k$, similar to the proof of Lemma~\ref{prop:kef}.
\begin{lemma}[Large Actual Reduction] \label{lem:i}
  Suppose Assumptions~\ref{a:p} and~\ref{a:m1} hold.
Then on $\mathcal{I}_k \cap \mathcal{B}_k$,
\begin{equation*}
-\aredr_k\le -c_4\|h(x_k)\|\delta_k < 0,
\quad\text{where}\quad
c_4 \coloneqq \kfcd- \left( 2\kef + \kfcd\kgrad\right)\zeta^{-1}.
\end{equation*}
\end{lemma}
\begin{proof}
Per the proof of Lemma~\ref{prop:kef}, the following pointwise bound holds on $\mathcal{I}_k$:
\begin{equation}\label{eq:ikb1}
-\aredr_k = \predr_k - \aredr_k - \predr_k\le 2\kef\delta_k^2 - \predr_k.
\end{equation}
From the fraction of Cauchy decrease condition (S2) then,
\begin{equation*}
-\aredr_k\le 2\kef\delta_k^2 - \kfcd\|h_k\|\min\left\{\frac{\|h_k\|}{\kbmh},\delta_k\right\},
\end{equation*}
while from the definitions of $\mathcal{I}_k$ and $\mathcal{B}_k$,
$$\|h_k\|\ge \|h(x_k)\| - \|h_k-h(x_k)\| \ge (\zeta - \kgrad)\delta_k \ge 4\kef\delta_k \ge \kbmh\delta_k.$$
It follows that $\delta_k$ is the smaller of the two quantities in the minimum bounding $-\aredr_k$, so
\begin{align*}
-\aredr_k
&\le 2\kef\delta_k^2 - \kfcd\left(\|h(x_k)\| - \kgrad\delta_k\right)\delta_k \\
&\le -[\kfcd - \zeta^{-1} (2\kef+\kfcd\kgrad)]\|h(x_k)\|\delta_k.
\end{align*}
The term in square brackets is $c_4$, which is positive due to the definition of $\zeta$.
\end{proof}

The second lemma will be key to the entire analysis.
For the lemma, it suffices to relax the definition of ``an accurate model and a small trust-region radius'' from $\mathcal{I}_k\cap\mathcal{B}_k$ to $\mathcal{I}_k\cap\mathcal{R}_k$, where
$$\mathcal{R}_k := \{\|H_k\|\ge (\zeta - \kgrad)\Delta_k\}.$$
This change is indeed a relaxation since, as seen in the proof of Lemma~\ref{lem:i}, the former event implies the latter.

\begin{lemma}[Successful Step]\label{lem:step}
  Suppose Assumptions~\ref{a:p} and~\ref{a:m1} hold.
Then $\mathcal{I}_k\cap \mathcal{J}_k\cap \mathcal{R}_k$ (i.e., the event of an accurate model, an accurate computed reduction, and a small trust-region radius) is contained in $\mathcal{S}_k$ (acceptance of the trial step).
\end{lemma}

  The proof of Lemma~\ref{lem:step} relies on $\predr_k$ being at least as large as a constant times $\|h_k\|\delta_k$ on $\mathcal{I}_k\cap\mathcal{R}_k$, similar to the proof Lemma~\ref{lem:i}. Lemma~\ref{lem:step} also requires an accurate $\credr_k$ due to the first step acceptance condition in \eqref{eq:acceptance1}.

\begin{proof}
We verify the two trial step acceptance criteria in \eqref{eq:acceptance1}.
On $\mathcal{R}_k$, $\|h_k\| \ge \eta_2\delta_k$, which is one of the two trial step acceptance conditions.
The other is
$$
  \rho_k := \frac{\credr_k}{\predr_k} \ge \eta_1.
$$
To see that this condition holds, first observe that from the triangle inequality and $\predr_k>0$,
$$|\rho_k - 1|\predr_k \le |\credr_k - \aredr_k| + |\aredr_k - \predr_k|.$$
We divide this expression by $\predr_k$ and note that on $\mathcal{I}_k\cap \mathcal{J}_k$,
$$
|\aredr_k-\predr_k|\le 2\kef\delta_k^2
\quad\text{and}\quad
|\credr_k-\aredr_k|\le\eta\predr_k
$$
which yields
$$|\rho_k - 1| \le \eta + \frac{2\kef\delta_k^2}{\predr_k}.$$
Since $\predr_k$ is at least $\kfcd\|h_k\|\delta_k$ on $\mathcal{R}_k$,
$$|\rho_k - 1| \le \eta + \frac{2\kef\delta_k}{\kfcd\|h_k\|},$$
but also on $\mathcal{R}_k$, $\delta_k\le \zeta^{-1}\|h_k\| \le (1 - \eta_1 -\eta)\kfcd/(2\kef)$, so
$$|\rho_k - 1| \le \eta + (1 - \eta_1 - \eta),$$
which implies the second step acceptance condition holds.
\end{proof}

Let
$\Psi_k \coloneqq \Psi_k(X_k,\Delta_k) = \nu ( \sobj(X_k) + \nobj(X_k) ) + (1-\nu) \Delta_k^2,$
where $\nu\in(0,1)$ is a constant such that
\begin{equation}\label{eq:nu}
  \quad \frac{\nu}{1-\nu}
>
\frac{\gamma^2 - \gamma^{-2}}{(\eta_1-\eta)\min\{\kfcd,1\} }\max\left\{\frac{\kbmh}{\eta_2},1\right\}.
\end{equation}
We will use the preceding lemmas to show that the conditional expectation of $\dpsi$ is negative (i.e., $\Psi_k$ decreases on average) and that this decrease is at least as large as a constant times the trust-region radius squared.
Similar to the analysis of STORM, our strategy will be to partition the probability space $\Omega$ into events.
Thus, it is useful to first bound the change in $\Psi_k$ in some important settings.
The first of these settings is $\overline{\mathcal{S}_k}$---the event that the trial step is rejected.
It follows immediately from the definition of $\psi_k = \Psi_k(\omega)$ that on $\overline{\mathcal{S}_k}$,
\begin{equation}\label{eq:rej}
  \psi_{k+1}-\psi_k \le -c_5\delta_k^2 < 0,
\quad\text{where}\quad
    c_5 \coloneqq (1-\nu)\left(1 - \gamma^{-2}\right).
\end{equation}
The next proposition is that this bound holds on the entirety of $\mathcal{J}_k$.
\begin{proposition}[Accurate Computed Reduction $\Rightarrow$ $\Psi_k$ Decrease]\label{prop:jkb}
  Suppose Assumption~\ref{a:m1} holds.
Then \eqref{eq:rej} holds on $\mathcal{J}_k$.
\end{proposition}
\begin{proof}
The bound \eqref{eq:rej} holds on $\overline{\mathcal{S}_k}\supseteq \mathcal{J}_k\cap\overline{\mathcal{S}_k}$.
To see that it holds on $\mathcal{J}_k\cap \mathcal{S}_k$, note that on $\mathcal{S}_k$,
$$-\aredr_k
\le |\credr_k-\aredr_k|-\credr_k
\le |\credr_k-\aredr_k|-\eta_1\predr_k,
$$
where the second inequality is a consequence of the step acceptance condition $\credr_k\ge\eta_1\predr_k$.
Combining the other step acceptance condition, $\|h_k\|\ge\eta_2\delta_k$, with
$$|\credr_k-\aredr_k|\le\eta\predr_k,$$
which holds on $\mathcal{J}_k$, we thus have from the fraction of Cauchy decrease condition (S2),
\begin{equation*}
-\aredr_k
\le -(\eta_1-\eta)\predr_k \le -(\eta_1-\eta)\kfcd\min\left\{\frac{\eta_2}{\kbmh},1\right\}\delta_k^2 < 0.
\end{equation*}
Consequently,
\begin{equation*}
\psi_{k+1}-\psi_k
\le \left[-\nu(\eta_1-\eta)\kfcd\min\left\{\frac{\eta_2}{\kbmh},1\right\} + (1-\nu)(\gamma^2-1)\right]\delta_k^2,
\end{equation*}
so by the definition of $\nu$---i.e., \eqref{eq:nu}---the bound \eqref{eq:rej} holds on $\mathcal{J}_k\cap\mathcal{S}_k$.
\end{proof}

We also bound the change in $\Psi_k$ pointwise on the other ``good'' event $\mathcal{I}_k$.
The bound is positive, and hence weaker than \eqref{eq:rej}.
\begin{proposition}[Accurate Model $\Rightarrow$ Bounded $\Psi_k$ Increase]\label{prop:ikb}
  Suppose Assumptions~\ref{a:p} and~\ref{a:m1} hold.
Then on $\mathcal{I}_k$,
\begin{equation}\label{eq:ikb}
\psi_{k+1}-\psi_k \le c_6\delta_k^2,
\quad\text{where}\quad
c_6 \coloneqq 2\nu\kef + (1-\nu)(\gamma^2 -1).
\end{equation}
\end{proposition}
\begin{proof}
The proposition follows from \eqref{eq:ikb1}: the $-\predr_k$ term there is negative, so dropping it yields the desired result.
\end{proof}

Lemma~\ref{lem:i} can be used to tighten this positive bound on $\mathcal{I}_k\cap\mathcal{B}_k\cap\mathcal{S}_k$ into guaranteed decrease.

\begin{corollary}[Guanteed $\Psi_k$ Decrease]\label{cor:bkc}
 Suppose the assumptions of Lemma~\ref{lem:i} hold.
Then on $\mathcal{I}_k\cap\mathcal{B}_k\cap\mathcal{S}_k$,
$$\psi_{k+1}-\psi_k \le -\nu c_4\|h(x_k)\|\delta_k + (1-\nu)(\gamma^2 - 1)\delta_k^2 < 0.$$
\end{corollary}
\begin{proof}
The first inequality is a consequence of Lemma~\ref{lem:i} and the definition of $\psi_k$.
To show that the bound is negative, we establish
\begin{equation}\label{eq:c7zeta}
-\frac{1}{2}\nu c_4\zeta + (1-\nu)(\gamma^2-1) < 0,
\end{equation}
which is an identity we use later as well.
This identity is sufficient since on $\mathcal{B}_k$,
$$
-\nu c_4\|h(x_k)\|\delta_k + (1-\nu)(\gamma^2 - 1)\delta_k^2
\le
[-\nu c_4+(1-\nu)(\gamma^2-1)\zeta^{-1}] \|h(x_k)\|\delta_k,
$$
which is bounded above by \eqref{eq:c7zeta}.
We note that 1/4 is bounded below by the function $a(1-a)$ defined for $a\in\R$; thus, from the definitions of $c_4$, $\zeta$, and $\kef$,
\begin{multline*}
\frac{1}{2}c_4\zeta = \frac{1}{2}\kfcd\left(\zeta-\kgrad-\frac{2\kef}{\kfcd}\right)
\ge \frac{1/4}{1-\eta_1-\eta}\\
\ge \frac{\eta_1(1-\eta_1) - \eta(1-\eta)}{1-\eta_1-\eta}
= \frac{(\eta_1-\eta)(1-\eta_1-\eta)}{1-\eta_1-\eta}
= (\eta_1-\eta).
\end{multline*}
This identity together with the definition of $\nu$ establishes \eqref{eq:c7zeta}.
\end{proof}

Our final preliminary result concerns the change in $\Psi_k$ on $\overline{\mathcal{I}_k\cup\mathcal{J}_k} = \overline{\mathcal{I}_k}\cap\overline{\mathcal{J}_k}$, i.e., the ``worst case'' that neither $\mathcal{I}_k$ nor $\mathcal{J}_k$ occur.
The result, which is primarily a consequence of Assumption~\ref{a:new}, is in a sense weaker than the results that precede as it is in expectation, as opposed to pointwise.

\begin{proposition}[Inaccurate Model and Inaccurate Computed Reduction $\Rightarrow$ Expected $\Psi$ Increase]\label{prop:newb}
  Suppose Assumptions~\ref{a:m2}--\ref{a:new} hold.
Then
\begin{multline*}
  \mathbb{E}\left[(\dpsi)\mathbbm{1}(\overline{\mathcal{I}_k}\cap\overline{\mathcal{J}_k})\Big\vert \filt_{k-1}\right] \\
\le (1-\alpha)\Big(\nu c_2\|h(X_k)\|\Delta_k  + \left[\nu c_1 + (1-\beta)(1-\nu)(\gamma^2 - 1)\right]\Delta_k^2\Big).
\end{multline*}
\end{proposition}
\begin{remark}
The quantity being bounded resembles an average (also called conditional) value at risk \cite{rockafellar.2000}.
The quantity is the conditional expectation of $\Psi_{k+1} - \Psi_k$ on $\overline{\mathcal{I}_k}\cap\overline{\mathcal{J}_k}$, which is a region in probability space that belongs to the tails of both $\|\nabla M_k(X_k) - \nabla f(X_k)\|$ and $|\cred_k -\ared_k|$.
\end{remark}
\begin{proof}
  From the definition of $\Psi_k$,
\begin{multline}\label{eq:wc}
  \mathbb{E}\left[(\dpsi)\mathbbm{1}(\overline{\mathcal{I}_k}\cap\overline{\mathcal{J}_k})\Big\vert \filt_{k-1}\right] \\
  \le \nu \mathbb{E}\left[-\ared_k\mathbbm{1}(\mathcal{S}_k\cap\overline{\mathcal{I}_k}\cap\overline{\mathcal{J}_k})\Big\vert \filt_{k-1}\right]
  + (1-\nu)(\gamma^2-1)\Delta_k^2\mathbb{E}\left[\mathbbm{1}(\overline{\mathcal{I}_k}\cap\overline{\mathcal{J}_k})\Big\vert \filt_{k-1}\right],
\end{multline}
wich uses $\Delta_k\in \mathcal{F}_{k-1}$
The tower property simplifies the right-hand side of \eqref{eq:wc}.
In particular,
\begin{equation*}
\mathbb{E}\left[\mathbbm{1}(\overline{\mathcal{I}_k}\cap\overline{\mathcal{J}_k})\Big\vert \filt_{k-1}\right]
=\mathbb{E}\left[\mathbbm{1}(\overline{\mathcal{I}_k})\mathbb{E}\left[\mathbbm{1}(\overline{\mathcal{J}_k})\Big\vert \filt_{k-1/2}\right]\Big\vert \filt_{k-1}\right]
\le (1-\alpha)(1-\beta).
\end{equation*}
and
\begin{align*}
\mathbb{E}\left[-\ared_k\mathbbm{1}(\mathcal{S}_k\cap\overline{\mathcal{I}_k}\cap\overline{\mathcal{J}_k})\Big\vert \filt_{k-1}\right]
&=\mathbb{E}\left[\mathbbm{1}(\overline{\mathcal{I}_k})\mathbb{E}\left[-\ared_k\mathbbm{1}(\mathcal{S}_k\cap\overline{\mathcal{J}_k})\Big\vert \filt_{k-1/2}\right]\Big\vert \filt_{k-1}\right] \\
&\le (1-\alpha)\left(c_1\Delta_k^2  + c_2\|h(X_k)\|\Delta_k\right),
\end{align*}
where the inequality is implied by Assumption~\ref{a:new}.
\end{proof}

The preceding results position us to state and prove the main result of this subsection.
\begin{lemma}[Expected $\Psi$ Decrease]\label{lem:incs}
For any $\Theta$ satisfying
\begin{equation}\label{eq:tbounds}
  \frac{1}{2}(1-\nu)(1-\gamma^{-2})\le \Theta <(1-\nu)(1-\gamma^{-2}),
\end{equation}
there exist
$0 \le \alpha,\beta < 1$ 
such that if Assumptions~\ref{a:p}--\ref{a:new} hold with these $\alpha$ and $\beta$,
then for $k\in\mathbb{N}$, 
\begin{equation}\label{eq:incs}
  \mathbb{E}\left[\Psi_{k + 1} - \Psi_k\Big|\filt_{k-1}\right]\le -\Theta\Delta_k^2.
\end{equation}
\end{lemma}

\begin{proof}
Similar to the proof of the analogous result in STORM, we use the event $\mathcal{B}_k$ and its complement, $\overline{\mathcal{B}_k}$, to partition $\Omega$.
We write 
\begin{equation}\label{eq:12}
\mathbb{E}\left[\dpsi \Big\vert\filt_{k-1}\right]
= \underbrace{\mathbb{E}\left[(\dpsi)\mathbbm{1}(\mathcal{B}_k) \Big\vert\filt_{k-1}\right]}_{\normalfont\Circled{1}}
+ \underbrace{\mathbb{E}\left[(\dpsi)\mathbbm{1}(\overline{\mathcal{B}_k}) \Big\vert\filt_{k-1}\right]}_{\normalfont\Circled{2}},
\end{equation}
and apply the results prior to this lemma to bound each term.
Term \Circled{2} is the easier one.
We partition $\overline{\mathcal{B}_k}$ into the disjoint sub-events
\begin{equation}\label{eq:sets}
\overline{\mathcal{B}_k}\cap \mathcal{J}_k, \quad \overline{\mathcal{B}_k}\cap\mathcal{I}_k\cap \overline{\mathcal{J}_k},\quad\text{and}\quad \overline{\mathcal{B}_k}\cap\overline{\mathcal{I}_k}\cap\overline{\mathcal{J}_k},
\end{equation}
on which we apply Propositions~\ref{prop:jkb},~\ref{prop:ikb}, and~\ref{prop:newb}, respectively.
In particular, for $\overline{\mathcal{B}_k}\cap\overline{\mathcal{J}_k}\cap\overline{\mathcal{I}_k}$,
\begin{multline*}
\mathbb{E}\left[(\dpsi)\mathbbm{1}(\overline{\mathcal{B}_k}\cap\overline{\mathcal{I}_k}\cap\overline{\mathcal{J}_k})\Big\vert\filt_{k-1}\right]
=   \mathbbm{1}(\overline{\mathcal{B}_k}) \mathbb{E}\left[(\dpsi)\mathbbm{1}(\overline{\mathcal{I}_k}\cap\overline{\mathcal{J}_k})\Big\vert\filt_{k-1}\right]\\
\le \mathbbm{1}(\overline{\mathcal{B}_k})(1-\alpha)\Big(\nu c_2\|h(X_k)\|\Delta_k  + [\nu c_1 + (1-\beta)(1-\nu)(\gamma^2 - 1)]\Delta_k^2\Big).
\end{multline*}
Recall that $\|h(x_k)\| < \zeta\delta_k$ on $\overline{\mathcal{B}_k}$, so this right-hand side is bounded above by
\begin{equation*}
\mathbbm{1}(\overline{\mathcal{B}_k})(1-\alpha)
[\nu(c_1 + c_2\zeta) + (1-\beta)(1-\nu)(\gamma^2 - 1)]\Delta_k^2.
\end{equation*}
Combining this bound with Propositions~\ref{prop:jkb} and~\ref{prop:ikb} applied to the other sets in \eqref{eq:sets} results in
\begin{equation}\label{eq:bkc}
\Circled{2}
\le
\mathbbm{1}(\overline{\mathcal{B}_k})
\Big(- \beta c_5
      + (1-\beta) c_6
      + (1-\alpha)[\nu(c_1 + c_2\zeta) + (1-\beta)(1-\nu)(\gamma^2 - 1)]
\Big)\Delta_k^2,
\end{equation}
which can be made negative by choosing $\alpha$ and $\beta$ sufficiently close to one.
Now we bound \Circled{1}.
A bound proportional to $\Delta_k^2$ is challenging for \Circled{1} because in the ``worst case'' $\mathcal{B}_k\cap\overline{\mathcal{I}_k}\cap\overline{\mathcal{J}_k}$,
Proposition~\ref{prop:newb} gives
\begin{multline}\label{eq:bk}
\mathbb{E}\Big[(\Psi_{k+1}-\Psi_k)\mathbbm{1}(\mathcal{B}_k\cap\overline{\mathcal{I}_k}\cap\overline{\mathcal{J}_k})\Big\vert\filt_{k-1}\Big] \\
\le \mathbbm{1}(\mathcal{B}_k)(1-\alpha)\left(\nu\left[\frac{c_1}{\zeta} + c_2\right]\|h(X_k)\|\Delta_k  + [(1-\beta)(1-\nu)(\gamma^2 - 1)]\Delta_k^2\right).
\end{multline}
This bound has a positive $\|h(X_k)\|\Delta_k$ term, and on $\mathcal{B}_k$, $\|h(x_k)\|\delta_k$ is large relative to $\delta_k^2$.
Our strategy is to offset the term with a negative $\|h(X_k)\|\Delta_k$ term from the ``best case'' $\mathcal{B}_k\cap\mathcal{I}_k\cap\mathcal{J}_k$.
We have from Lemma~\ref{lem:step}, that $\mathcal{B}_k\cap\mathcal{I}_k\cap\mathcal{J}_k\subseteq \mathcal{B}_k\cap\mathcal{I}_k\cap\mathcal{S}_k$, so from Corollary~\ref{cor:bkc},
\begin{multline*}
\mathbb{E}\left[(\dpsi)\mathbbm{1}(\mathcal{B}_k\cap\mathcal{I}_k\cap \mathcal{J}_k) \Big\vert\filt_{k-1}\right] \\
\le
\mathbbm{1}(\mathcal{B}_k)\alpha\beta\left(-\nu c_4\|h(X_k)\|\Delta_k + (1-\nu)(\gamma^2 - 1)\Delta_k^2\right).
\end{multline*}
Combining this result with \eqref{eq:bk} yields
\begin{align}\label{eq:bk2}
&\mathbb{E}\left[(\dpsi)[\mathbbm{1}(\mathcal{B}_k\cap\mathcal{I}_k\cap \mathcal{J}_k) +\mathbbm{1}(\mathcal{B}_k\cap\overline{\mathcal{I}_k}\cap\overline{\mathcal{J}_k})] \Big\vert\filt_{k-1}\right]
\notag\\
&\le
\mathbbm{1}(\mathcal{B}_k)\left(
-\nu\left[\alpha\beta c_4 - (1-\alpha)\left(\frac{c_1}{\zeta} + c_2\right) \right]\|h(X_k)\|\Delta_k + (1-\nu)(\gamma^2 - 1)\Delta_k^2
\right),
\end{align}
where we have simplified the coefficient of the $\Delta_k^2$ term by noting that since $\alpha$ and $\beta$ lower bound probabilities,
$$(1 - \alpha)(1-\beta) + \alpha\beta = 1 + \alpha(\beta - 1) + \beta(\alpha - 1)\le 1.$$
What remains of $\mathcal{B}_k$ outside of $(\mathcal{I}_k\cap\mathcal{J}_k)\cup (\overline{\mathcal{I}_k}\cap\overline{\mathcal{J}_k})$ belongs to $\mathcal{I}_k\cap\overline{\mathcal{J}_k}$ or $\overline{\mathcal{I}_k}\cap\mathcal{J}_k$.
On both of these events, $\psi_{k+1}-\psi_k\le 0$, due to \eqref{eq:rej} and Corollary~\ref{cor:bkc} in the case of the former and Proposition~\ref{prop:jkb} in the case of the latter.
We thus conclude that the right-hand side of \eqref{eq:bk2} bounds \Circled{1}.

To complete the proof of Lemma~\ref{lem:incs}, we must show that our bounds for \Circled{1} and \Circled{2} sum to $\Theta\Delta_k^2$ for a $\Theta$ that satisfies \eqref{eq:tbounds}.
We accomplish this by choosing the probability lower bounds $\alpha$ and $\beta$ for accurate models and accurate computed reductions, respectively, to be sufficiently close to one.
As a convenience to the reader, Table~\ref{tbl:c} recapitulates the constants pertinent to the analysis up to this point. 
We omit $c_3$ since it does not enter the analysis directly.
First, we choose $\alpha$ and $\beta$ sufficiently close to one that
\begin{equation}\label{eq:p1}
  \frac{\alpha\beta-\frac{1}{2}}{1-\alpha} > \frac{c_1+c_2\zeta}{c_4\zeta}.
\end{equation}
This is possible even in the smooth case where $c_1,c_2\propto (1-\beta)$.
We claim that this choice of $\alpha$ and $\beta$ implies
\begin{equation}\label{eq:bk4}
  \Circled{1}\le-\mathbbm{1}(\mathcal{B}_k) (1-\nu)(1-\gamma^{-2})\Delta_k^2.
\end{equation}
To explain why, we rearrange \eqref{eq:p1} into
\begin{equation*}
\frac{1}{2}c_4 < \alpha\beta c_4 - (1-\alpha)\left(\frac{c_1}{\zeta}+c_2\right)
\end{equation*}
and recall that the right-hand side of \eqref{eq:bk2} bounds \Circled{1}. 
As a result,
\begin{equation*}
\Circled{1}
\le
\mathbbm{1}(\mathcal{B}_k)\left(
-\frac{1}{2}\nu c_4\zeta + (1-\nu)(\gamma^2 - 1),
\right)\Delta_k^2.
\end{equation*}
which implies \eqref{eq:bk4} when combining \eqref{eq:c7zeta} with the definition of $\nu$ and a bit of algebra.
Second, we choose $\alpha$ and $\beta$ sufficiently close to one that 
\begin{equation}\label{eq:p2}
  \beta(1-\nu)(1-\gamma^{-2}) - \Theta 
  \ge (1-\beta)c_6
+ (1-\alpha)[\nu(c_1 + c_2\zeta) + (1-\beta)(1-\nu)(\gamma^2 - 1)].
\end{equation}
This is always possible since in the limit of $\alpha,\beta\to 1$, \eqref{eq:p2} is $(1-\nu)(1-\gamma^{-2}) -\Theta \ge 0$, which is strict by assumption.
It then follows from \eqref{eq:bkc} that
\begin{equation}\label{eq:bkc2}
\Circled{2} \le -\mathbbm{1}(\overline{\mathcal{B}_k})\Theta\Delta_k^2.
\end{equation}
From \eqref{eq:bk4} and \eqref{eq:bkc2}, we conclude
\begin{equation*}
  \mathbb{E}\Big[\dpsi \Big\vert\filt_{k-1}\Big] = \Circled{1} + \Circled{2} \le -\Theta\Delta_k^2,
\end{equation*}
which proves Lemma~\ref{lem:incs}.
\end{proof}

\begin{table}[h!]
\caption{The constants $c_k$ in the proof of Lemma~\ref{lem:incs}.
}\label{tbl:c}
\begin{tabular}{@{}p{\textwidth}@{}}
\centering
\begin{tabular}{c|c}\hline
\multirow{2}{*}{$c_1$} & \parbox{7.5cm}{\raggedright $\Delta_k^2$ coefficient in Assumption~\ref{a:new}; } \\
                       & \parbox{7.5cm}{\raggedleft  $\quad$ can be set to $\frac{L}{2}(1-\beta)$ when $\nobj\equiv 0$ (see Proposition~\ref{prop:new})}
\\[0.5em]
\hline
\\[-0.5em]
\multirow{2}{*}{$c_2$} & \parbox{7.5cm}{\raggedright $\|h(X_k)\|\Delta_k$ coefficient in Assumption~\ref{a:new}; } \\
                       & \parbox{7.5cm}{\raggedleft $\quad$ can be set to $(1-\beta)$ when $\nobj\equiv 0$ (see Proposition~\ref{prop:new})}
\\[0.5em]
\hline
\\[-0.5em]
$c_5$ &  {\centering $(1-\nu)\left(1 - \gamma^{-2}\right)$}
\\[0.5em]
\hline
\\[-0.5em]
$c_6$ & $2\nu\kef + (1-\nu)(\gamma^2 -1)$
\\[0.5em]
\hline
\\[-0.5em]
$c_4$ & $\kfcd- \left( 2\kef + \kfcd\kgrad\right)\zeta^{-1}$
\end{tabular}
\end{tabular}
\end{table}

Having proved Lemma~\ref{lem:incs}, we see that the following assumption plays a key role in the analysis.
\begin{assumption}\label{a:ab}
  The probability lower bounds $\alpha$ and $\beta$ that appear in Assumptions~\ref{a:m2} and~\ref{a:c} satisfy \eqref{eq:p1} and \eqref{eq:p2}.
\end{assumption}
Remarks concerning the different quantities in our analysis are warranted.
\begin{remark}[Deterministic Computed Reduction]
The limit $\beta\to 1$ corresponds to the computed reduction being accurate with probability one.
That is, $P(\overline{\mathcal{J}_k}) = 0$ for all $k$, in which case we can take $c_1=c_2=0$ in Assumption~\ref{a:new}.
Condition \eqref{eq:p2} is then satisfied with $\Theta = (1-\nu)(1-\gamma^{-2})$ and \eqref{eq:p1} reduces to $\alpha\ge 1/2$, i.e., to models being more accurate more often than not.
\end{remark}

\begin{remark}[Computing $\alpha$ and $\beta$]
Though $\alpha$ and $\beta$ that satisfy Assumption~\ref{a:ab} always exist, it can be difficult to determine numerical values for them since \eqref{eq:p1} and \eqref{eq:p2} depend on quantities like the Lipschitz constant of $\nabla f$ and the model Hessian bound.
These values can be viewed as guidelines for understanding what can prevent convergence.
Nevertheless, we did not find it a difficult practical issue in our computational results. 
\end{remark}

\begin{remark}[Lipschitz Constant]
Consider the limit of $L$, the Lipschitz constant of $\nabla f$, diverging to $\infty$.
The quantity $\zeta$ is proportional
proportional to $L$.
(See \eqref{eq:zeta} and \eqref{eq:kef} for the definition of $\zeta$.)
It follows that right-hand side of \eqref{eq:p1} converges to a constant as $L\to\infty$, in which case that restriction on $\alpha$ and $\beta$ approaches a finite number.
In contrast, the growth of the right-hand side of \eqref{eq:p2} is unbounded as $L\to\infty$, which drives $1-\beta$ and $(1-\alpha)c_2$ to zero.
Thus, as $L$ increases the computed reductions must be sufficiently accurate often because otherwise the increase from a poor step could be large.
\end{remark}

\begin{remark}[Complexity]
We will see in Section~\ref{sec:co} that the complexity of ProxSTORM is inversely proportional to $\Theta$, which itself is proportional to $1-\nu$.
The complexity is also proportional to $\zeta^2$.
Thus, we have the following.
\begin{itemize}
\item
An increase in $\zeta$---be it through $\kef, \kgrad$ or $\kbmh$ (model parameters), $\eta$ (a computed reduction parameter), or even $\eta_1$ (a step acceptance parameter)---adversely affects the complexity of ProxSTORM.
\item
The difference $1-\nu$ (and hence $\Theta$) is asymptotically proportional to $\eta_1\eta_2$; see \eqref{eq:nu}.
In other words, if the algorithm accepts steps using more relaxed conditions it may make less progress per step.
\end{itemize}
\end{remark}

\begin{remark}[Prior Work]
The parameters $\zeta$ and $\nu$ differ slightly from their definitions in \cite{storm,storm-rates}, stemming, in part, from our alternative definition of $\mathcal{J}_k$.
\end{remark}

To provide a sense of reasonable values for our parameters, we note that the numerical results in Section~\ref{sec:ex} use $\eta_1 = \frac{1}{2}$ and $\gamma = 5$.

\subsection{First-Order Global Convergence}\label{sec:gc}

We begin this subsection with the an overview of ProxSTORM's global convergence.
On a realization of the algorithm $\omega$ such that $\lim_{k\to\infty}\delta_k = 0$ but $\liminf_{k\to\infty} \|h(x_k)\|\neq 0$, we have that for any constant $c$, there exists a number $K$ dependent on $c$ such that
$$
  \|h(x_k)\| \ge c\delta_k \quad \text{for all} \quad k > K.
$$
For $c \ge \zeta$ and $k > K$, $\omega\in \mathcal{I}_k\cap \mathcal{J}_k$ implies acceptance of the step $s_k$
(see Lemma~\ref{lem:step}).
Acceptance of the step increases the trust-region radius
as per $\delta_{k+1} = \gamma\delta_k,$
but we will show $\omega\in \mathcal{I}_k\cap \mathcal{J}_k$ for enough $k$ that $\lim_{k\to\infty}\delta_k \neq 0$; that is,
\begin{equation}\label{eq:disjoint}
  \mathbb{P}\left(\left\{\lim_{k\to\infty}\Delta_k = 0\right\} \cap \left\{\liminf_{k\to\infty}\|h(X_k)\| \neq 0\right\} \right) = 0.
\end{equation}
Our first result in this subsection is to show that the second event in \eqref{eq:disjoint} has probability one.
This will imply that if indeed \eqref{eq:disjoint} holds, the first event (a lack of global convergence) must have probability zero.

\begin{corollary}[Trust-Region Radius Converges to Zero]\label{cor:tozero}
  Suppose the assumptions of Lemma~\ref{lem:incs} hold.
Then the ProxSTORM trust-region radius, $\Delta_k$, satisfies
\begin{equation*}\label{eq:tozero}
  \lim_{k\to\infty} \Delta_k = 0
\quad\text{with probability one}.
\end{equation*}
\end{corollary}
\begin{proof}
Taking the expectation of \eqref{eq:incs} and summing over $k$ gives
$$
  \Theta \mathbb{E}\left[\sum_{k=1}^{n} \Delta_k^2\right] \le \Psi_0 - \mathbb{E}[\Psi_{n+1}].
$$
Since both $\{\Psi_k\}_k$ and $\{\Delta_k\}_k$ are uniformly bounded from below, the monotone convergence theorem implies
\begin{equation*}\label{eq:essum}
  \mathbb{E}\left[\sum_{k=1}^\infty \Delta_k^2\right] = \lim_{n\to\infty} \mathbb{E}\left[\sum_{k=1}^n\Delta_k^2\right] < \infty.
\end{equation*}
It follows from the definition of the expectation that this is possible only if $\sum_{k=1}^\infty\Delta_k^2$ is finite on sets of positive probability, which completes the proof.
\end{proof}

To make our overview of the global convergence proof precise, we must relate our reasoning about individual realizations of the algorithm to the frequency of $\mathcal{I}_k\cap \mathcal{J}_k$, which is formulated in terms of probabilities, i.e., in terms of the ensemble of ProxSTORM realizations.
We navigate this technicality with martingales \cite{durrett5}.

\begin{definition}
A sequence of integrable random variables, $Y_1$, $Y_2$, $\ldots$ is a \emph{submartingale} with respect to a filtration $\mathcal{G}_1 \subseteq \mathcal{G}_2 \subseteq \ldots$ if, for all $k$,
$$
  \text{$Y_k$ is measurable with respect to $\mathcal{G}_k$}
  \quad\text{and}\quad
  E[Y_{k+1}|\mathcal{G}_k] \ge Y_k.
$$
\end{definition}

\begin{theorem}[Submartingale Convergence]\label{thm:convergence}
  Suppose $Y_1$, $Y_2$, $\ldots$ is a submartingale with $Y_{k+1} - Y_k\le c < \infty$ for some constant $c$.
Then
\begin{equation}\label{eq:mconv}
  \mathbb{P}\left(\left\{\lim_{k\to\infty} Y_k \text{ exists and is finite}\right\} \cup \left\{\limsup_{k\to\infty} Y_k = \infty \right\}\right) = 1.
\end{equation}
\end{theorem}

\begin{proof}
  \cite[Exercise 4.2.4]{durrett5}.
\end{proof}

Similar to \cite[Theorem 4.16]{storm}, we use Theorem~\ref{thm:convergence} to prove global convergence.

\begin{theorem}[Global Convergence]\label{thm:liminf}
  Suppose the assumptions of Lemma~\ref{lem:incs} hold.
Then the ProxSTORM iterate, $X_k$, satisfies
\begin{equation*}
  \liminf_{k\to\infty} \|h(X_k)\| = 0 \quad\text{with probability one}.
\end{equation*}
\end{theorem}

\begin{proof}
  For $k\in\mathbb{N}$, let
\begin{equation*}
  W_k := \sum_{\ell = 0}^k \left[2\mathbbm{1}(\mathcal{I}_\ell\cap \mathcal{J}_\ell) - 1\right].
\end{equation*}
The sequence of $W_k$ is adapted to $\mathcal{F}_k$, and since $\alpha\beta > 1/2$, the sequence is a submartingale with respect to $\mathcal{F}_k$.
Concretely,
\begin{equation*}
  \mathbb{E}\left[W_{k+1}\Big|\mathcal{F}_k\right]
  = 2 \mathbb{E}\left[\mathbbm{1}(\mathcal{I}_{k + 1})E[\mathbbm{1}(\mathcal{J}_{k + 1}) | \mathcal{F}_{k + \frac{1}{2}}] \Big| \mathcal{F}_{k}\right] - 1 + W_k
  \ge 2 \alpha\beta - 1 + W_k
  \ge W_k.
\end{equation*}
Realizations of the submartingale have increments $W_k(\omega) - W_{k-1}(\omega) = \pm 1$, meaning $W_k$ cannot converge pointwise.
From Theorem~\ref{thm:convergence} then,
\begin{equation}\label{eq:liminf1}
  \mathbb{P}\left(\limsup_{k\to\infty}W_k = \infty\right) = 1.
\end{equation}
We use \eqref{eq:liminf1} to make the overview before Corollary~\ref{cor:tozero} precise.
Choose an $\omega$ in
\begin{equation}\label{eq:liminf4}
\left\{\lim_{k\to\infty}\Delta_k = 0\right\} \cap \left\{\liminf_{k\to\infty}\|h(X_k)\|\neq 0\right\},
\end{equation}
Since $\omega$ belongs to the latter event in \eqref{eq:liminf4}, there exists $\epsilon(\omega) > 0$ and $K_1(\omega)\in\mathbb{N}$ such that
$$
  \|h(X_k(\omega))\|\ge \epsilon(\omega) \quad\text{for all}\quad k>K_1(\omega).
$$
Let
$$
  \iota(\omega) := \min\left\{\frac{\overline{\delta}_{\max}}{\gamma},
                 \frac{\epsilon(\omega)}{2}\min\left\{\frac{1}{\eta_2},\frac{\min\{\kfcd,1\}(1-\eta_1-\eta)}{4\kef}\right\}\right\}.
$$
Since $\omega$ belongs to the former event in \eqref{eq:liminf4}, there exists a number $K_2(\omega)$ for which
$$
  \Delta_k(\omega) \le \iota(\omega) \quad\text{for all}\quad k>K_2(\omega).
$$
Let $k > K(\omega) := \max\{K_1(\omega),K_2(\omega)\}$ hereafter in this proof.
From the definition of $\iota$, we have that $\omega\in \mathcal{I}_k$ implies
\begin{align}\label{eq:liminf2}
  \|h(X_k(\omega))\| - \|H_k(\omega)\| &\le \|h(X_k(\omega)) - H_k(\omega)\|  \notag \\
  &\le \Big\|\nabla f(X_k(\omega)) - \nabla_x M_k(X_k(\omega);\omega)\Big\| \nonumber
  \le \kgrad\Delta_k(\omega) \le \frac{\epsilon(\omega)}{2}.
\end{align}
The second of these inequalities follows from \eqref{eq:ne}.
Since $\|h(X_k(\omega))\|\ge\epsilon(\omega)$, the inequalities rearrange into
\begin{equation*}
  \|H_k(\omega)\| \ge \frac{\epsilon(\omega)}{2} \ge \max\left\{\eta_2,\frac{4\kef}{\min\{\kfcd,1\}(1-\eta_1-\eta)}\right\} \Delta_k(\omega),
\end{equation*}
which means $\omega$ belongs to the event $\mathcal{R}_k$ defined before Lemma~\ref{lem:step}.
Our choice of $\omega$ therefore results in one of two possibilities for all $k>K(\omega)$:
\begin{enumerate}[i)]\small
\item $\omega\in \mathcal{I}_k\cap \mathcal{J}_k$, in which case Lemma~\ref{lem:step} ensures that ProxSTORM accepts $S_k(\omega)$.
Consequently, $\Delta_{k+1}(\omega) = \gamma\Delta_k(\omega)$, meaning
$$
  \log_{\gamma}\Delta_{k+1}(\omega) - \log_{\gamma}\Delta_{k}(\omega) = \log_{\gamma}\gamma = +1 = W_k(\omega) - W_{k-1}(\omega).
$$
\item $\omega\notin \mathcal{I}_k\cap \mathcal{J}_k$, in which case
$$
  \log_{\gamma}\Delta_{k+1}(\omega) - \log_{\gamma}\Delta_k(\omega) \ge \log_{\gamma}\frac{1}{\gamma} = -1 = W_k(\omega) - W_{k-1}(\omega).
$$
\end{enumerate}
In both of these possibilities, the increments of $\log_{\gamma}\Delta_{k+1}(\omega)$ upper bound the increments of $W_k(\omega)$.
Thus, there exists some constant $C(\omega)$ for which
$$
  \limsup_{k\to\infty} W_k(\omega) \le \limsup_{k\to\infty}\log_{\gamma}\Delta_{k+1}(\omega) + C(\omega).
$$
Since $\lim_{k\to\infty}\Delta_k(\omega) = 0$, it follows that $\limsup_{k\to\infty} W_k(\omega) = -\infty$.
Due to \eqref{eq:liminf1} then, we arrive at \eqref{eq:disjoint} by noting
\begin{align*}
  0 = \mathbb{P}\left(\limsup_{k\to\infty} W_k \neq \infty\right)
  &\ge \mathbb{P}\left(\limsup_{k\to\infty} W_k = -\infty\right)  \\
  &\ge \mathbb{P}\left(\left\{\liminf_{k\to\infty} \|h(X_k)\| \neq 0\right\}\cap\left\{\lim_{k\to\infty} \Delta_k = 0\right\}\right).
\end{align*}
Corollary~\ref{cor:tozero}, however, implies the second event in the intersection has probability one, so
\begin{equation*}
  \mathbb{P}\left(\liminf_{k\to\infty} \|h(X_k)\| \neq 0\right) = 0,
\end{equation*}
which completes the proof.
\end{proof}

We can strengthen Theorem~\ref{thm:liminf} into a limit-type result.

\begin{theorem}[Limit-Type Convergence]\label{thm:lim}
  Suppose the assumptions of Lemma~\ref{lem:incs} hold.
Then the ProxSTORM iterate, $X_k$, satisfies
\begin{equation*}
  \lim_{k\to\infty} \|h(X_k)\| = 0 \quad\text{with probability one}.
\end{equation*}
\end{theorem}
\begin{proof}
See Appendix~\ref{sec:app1}. The arguments are similar the proof of \cite[Theorem 4.18]{storm}.
\end{proof}

\subsection{Expected Complexity Bound}\label{sec:co}

Let $T_\epsilon$ be the iteration of ProxSTORM for which the true proximal gradient, $h(X_k)$, is first smaller than some threshold $\epsilon$, i.e., let
$$
  T_{\epsilon} := \inf_{\|h(X_k)\|\le \epsilon} k.
$$
Algorithm~\ref{alg:ps} and the definition of the filtration \eqref{eq:filt} imply that the event $\{T_\epsilon = k\}$ is measurable with respect to $\mathcal{F}_{k-1}$ (informally, $\mathcal{F}_{k-1}$ contains sufficient information to determine whether $\omega\in \{T_\epsilon = k\}$ or not).
Consequently, $T_\epsilon$ is a stopping time.
We apply \cite[Theorem 2]{storm-rates} to bound the expectation of $T_\epsilon$ and thus extend the expected complexity bound for STORM to ProxSTORM.
Specifically, we have the following.

\begin{theorem}[{\cite[Theorem 2]{storm-rates}}]\label{thm:co}
  Suppose the following conditions hold.
\begin{enumerate}
\item
The trust region radius $\Delta_k$ is uniformly bounded by some number $\delta_{\max} > 0$ independent of $k$ and $\omega$.

\item
There exists a constant $\lambda$ such that, for $k\in\mathbb{N}$, the trust region radius $\Delta_k$ satisfies
\begin{multline}\label{eq:rw}
  \Delta_{k+1}\one\left(T_{\epsilon}>k\right) \geq \min(\Delta_{k}e^{\lambda \widetilde{W}_{k}},\delta_{\epsilon})\one\left(T_{\epsilon }>k\right),\\
  \text{where}\quad
  \delta_\epsilon := \min\left\{\delta_0\gamma^{\left\lfloor\log_\gamma\left(\delta_0^{-1}\frac{\epsilon}{\zeta}\right)\right\rfloor},\delta_{\max}\right\}
\end{multline}
and $\widetilde{W}_{k}$ satisfies
\begin{equation*}
  \mathbb{P}(\widetilde{W}_{k}=1|\mathcal{F}_{k-1})  \geq \alpha\beta
  \quad\text{and}\quad
  \mathbb{P}(\widetilde{W}_{k}=-1|\mathcal{F}_{k-1}) \leq  1-\alpha\beta.
\end{equation*}

\item
There exists a positive constant $\Theta$ such that, for $k\in\mathbb{N}$, the function $\Psi_k$ satisfies
\begin{equation*}
{\mathbb{E}}[\Psi_{k+1}-\Psi_{k}|\mathcal{F}_{k}]\one\left(  T_{\epsilon}>k\right)
\leq -\Theta \Delta_{k}^2\one\left( T_{\epsilon}>k\right).
\end{equation*}
\end{enumerate}
Then
\begin{equation}\label{eq:complexity}
  \mathbb{E}[T_{\epsilon}] \le \frac{\alpha\beta}{2\alpha\beta - 1}\frac{\widetilde{\Psi}_0}{\Theta \delta_\epsilon^2} + 1
                           \le \frac{\alpha\beta}{2\alpha\beta - 1}\frac{\widetilde{\Psi}_0}{\Theta}\max\left\{\frac{\gamma\zeta}{\epsilon},\frac{1}{\delta_{\max}}\right\}^2 + 1,
\end{equation}
where $\widetilde{\Psi}_0 = \Psi_0 - \nu\inf_{x\in\R^d}\left\{\sobj(x) + \nobj(x)\right\}$.
\end{theorem}

The $\Psi_0$ in this theorem the function $\Psi_k$ at $k=0$, and the term subtracted from $\Psi_0$ is a lower bound for all possible values of $\Psi_k$, $k\in\mathbb{N}$, which is finite per Assumption~\ref{a:p}.
The constant $\zeta$ is defined in \eqref{eq:zeta}, and $\alpha$ and $\beta$ are the probability lower bounds in Assumptions~\ref{a:m2} and~\ref{a:c}.
Per Assumption~\ref{a:ab},
$$\alpha\beta > \frac{1}{2},$$
which is also required for our global convergence result (Theorem~\ref{thm:liminf}).
To see that the conditions of Theorem~\ref{thm:co} are satisfied by ProxSTORM, recall that
ProxSTORM has the maximal trust-region radius $\delta_{\max} = \delta_0\gamma^\ell$ and
the third condition in Theorem~\ref{thm:co} (\cite[Theorem 2]{storm-rates}) is Lemma~\ref{lem:incs}.
\begin{proposition}\label{p:step}
  Suppose the assumptions of Lemma~\ref{lem:incs} hold.
Then the second condition of Theorem~\ref{thm:co} is satisfied with $\lambda = \log\gamma$ and $\widetilde W_k = \mathbbm{1}(\mathcal{I}_k\cap\mathcal{J}_k) - 1$ and therefore ProxSTORM has the expected complexity \eqref{eq:complexity}. 
\end{proposition}
\begin{proof}
Let $\lambda = \log\gamma$.
The trust-region update rule of ProxSTORM implies that \eqref{eq:rw} holds when $\widetilde{W}_k = -1$.
When $\widetilde{W}_k = 1$, the inequality becomes
\begin{equation}\label{eq:co1}
  \Delta_{k+1}\one\left(T_{\epsilon}>k\right) \geq \min(\gamma\Delta_{k},\delta_{\epsilon})\one\left(T_{\epsilon }>k\right).
\end{equation}
We show that this expression holds on $\mathcal{I}_k\cap \mathcal{J}_k$, which is sufficient for the proposition to hold.

On the event $\{T_\epsilon \le k\}$, \eqref{eq:co1} is trivial.
On $\{T_\epsilon>k\}$, first consider $\{\Delta_k >\delta_\epsilon\}$.
Since both $\Delta_k$ and $\delta_\epsilon$ are the initial trust-region radius times an integer power of the parameter $\gamma$, we have that on $\{\Delta_k >\delta_\epsilon\}$,
$$\delta_{k+1}\ge \gamma^{-1}\delta_k \ge \delta_\epsilon.$$
Meanwhile, $\mathcal{I}_k\cap \{T_k > k\}\cap \{\Delta_k\le \delta_\epsilon\}\subseteq\mathcal{R}_k$, where $\mathcal{R}_k$ is the event defined before Lemma~\ref{lem:step}.
This is because $\epsilon\ge \zeta\delta_k$ on $\{\Delta_k\le \delta_\epsilon\}$, so on $\mathcal{I}_k\cap \{T_k > k\}\cap \{\Delta_k\le \delta_\epsilon\}$,
$$
  \|h_k\| \ge \|h(x_k)\| - \kgrad\delta_k > \epsilon - \kgrad\delta_k \ge (\zeta - \kgrad)\delta_k.
$$
From Lemma~\ref{lem:step} then,
$\mathcal{I}_k\cap \mathcal{J}_k\cap \{T_k > k\}\cap \{\Delta_k\le \delta_\epsilon\} \subseteq \mathcal{S}_k$, and on $\mathcal{S}_k$,
$$\delta_{k+1} \ge \gamma\delta_k.$$
Thus, \eqref{eq:co1} holds for
$$\{T_\epsilon \le k\}
\cup \Big(\{T_\epsilon > k\}\cap \{\Delta_k >\delta_\epsilon\}\Big)
\cup \Big(\mathcal{I}_k\cap \mathcal{J}_k \cap \{T_\epsilon > k\}\cap \{\Delta_k \le \delta_\epsilon\}\Big) \supseteq \mathcal{I}_k\cap\mathcal{J}_k,
$$
which completes the proof.
\end{proof}

\begin{remark}
Theorem~\ref{thm:co} establishes that the expected iteration complexity of ProxSTORM for achieving $\epsilon$-approximate stationary point is of the order of $\epsilon^{-2}$. 
This complexity matches that of STORM and well as that of the deterministic proximal trust region method in \cite{baraldi.2022}.
It also essentially matches all continuous optimization methods that rely only on first order information.
The other key constants in the complexity bound are $\alpha\beta$---the closer $\alpha\beta$ is to $\frac{1}{2}$ the larger $\frac{\alpha\beta}{2\alpha\beta - 1}$ is. 
See the remarks after the proof of Lemma~\ref{lem:incs} for a discussion on $\Theta$, $\zeta$ and other constants.
\end{remark}

\section{Oracle Assumptions and Total Sample Complexity}\label{sec:oracles}
Algorithm~\ref{alg:ps} relies on two stochastic oracles---the one that produces the model $m_k$ and the one that computes the computed reduction $\credr_k$.
The key stochastic quantity that defines the model is its gradient $\nabla M_k(X_k)=G_k$, which has to satisfy Assumption \ref{a:m2}.
The computed reduction has to satisfy  Assumption~\ref{a:c} and Assumption~\ref{a:new}.
In this section, we discuss these assumptions and how they can be ensured in some common settings.

We first consider the standard stochastic optimization and expected risk minimization setting where (with some abuse of notation)
\[
f(x)= \mathbb{E}_\xi [ f(x,\xi) ]
\]
with $f(x,\xi)$ denoting a random function dependent on the random variable $\xi$. 
One of the common and important assumptions made in stochastic optimization and machine learning literature is that
\begin{equation}\label{eq:expected_grad}
\nabla f(x)= \mathbb{E}_\xi [\nabla  f(x,\xi) ],
\end{equation}
in other words that integration and differentiation commute.
Suppose
$$\mathbb{E}_\xi [\|\nabla f(x)-\nabla f(x,\xi)\|^2]\le \sigma_g^2$$
uniformly in $x$.
The gradient approximation $g_k$ is usually computed by averaging sample gradients at $x_k$ on a minibatch of $n$ samples:
\[
g_k=\frac{1}{n} \sum_{\ell=1}^n \nabla  f(x_k,\xi_\ell).
\]
Clearly, $\mathbb{E}_\xi[g_k]=\nabla f(x_k)$ and $\mathbb{E}_\xi [\|\nabla f(x_k)-g_k\|^2]\le\frac{\sigma_g^2}{n}$, so a larger $n$ delivers a better approximation.
If for some $\tau$ we want
\[
\mathbb P_\xi\left(\|\nabla f(x)-g_k\|^2\leq \tau^2\right) = \mathbb{P}\left(\|\nabla M_k(X_k) - \nabla f(X_k)\|^2 \le \tau^2 |\filt_{k-1}\right)\ge \alpha,
\]
we have from Markov's inequality that any sample size
\[
n\geq \frac{\sigma_g^2}{\tau^2(1-\alpha)}
\]
achieves this bound.
When the Bernstein inequality is applicable (e.g., when $\nabla f(x,\xi)$ is Gaussian, or uniformly bounded), the bound on $n$ improves to
\[
n\geq \frac{\sigma_g^2}{\tau^2}\log\left(\frac{1}{1-\alpha}\right).
\]
In either case, choosing a sample size on each iteration to be proportional to $\delta_k^{-2}$ ensures Assumption \ref{a:m2} with a $\kgrad$ proportional to $\sigma_g$.
Next, we consider the computed reduction oracle.
We take $\nobj\equiv 0$, which is reasonable for understanding the computed reduction since ProxSTORM is based on $\nobj$ being easy to evaluate exactly.
Given two points $x_k$ and $x_k+s_k$, the actual reduction is $\aredr_k = f(x_k)-f(x_k+s_k)$ and the computed reduction $\credr_k$ can be calculated from a minibatch of samples again:
\[
\credr_k=\frac{1}{n} \sum_{\ell=1}^n \left( f(x_k,\xi_\ell)-  f(x_k+s_k,\xi_\ell)\right).
\]
The minibatch for the $\credr_k$ is chosen randomly and independently from the minibatch used in the computation of $g_k$.
We assume for the simplicity of the analysis that $f(x_k)$ and $f(x_k,\xi_\ell)$ are twice continuously differentiable with Lipschitz continuous second derivatives having Lipschitz constants bounded by some $\tilde{L}$.
Taylor expanding and applying the mean value theorem,
\begin{align*}
  f(x_k,\xi_\ell) - f(x_k+s_k,\xi_\ell) &= -\nabla f(x_k,\xi_\ell)^\top s_k - \frac{1}{2}s_k^\top\nabla^2 f(x_k+t_\ell s_k,\xi_\ell)s_k, && t_\ell\in (0,1), \\
  f(x_k)-f(x_k+s_k) &= -\nabla f(x_k)^Ts_k - \frac{1}{2}s_k^\top\nabla^2  f(x_k+ts_k,\xi)s_k,                                            && \;t\in(0,1).
\end{align*}
It follows that
\begin{multline*}
\aredr_k-\credr_k
= \left[ \frac{1}{n} \sum_{\ell=1}^n  \left(\nabla f(x_k,\xi_\ell)-\nabla f(x_k)\right)^\top \right ]s_k \\
+ \frac{1}{2}s_k^\top\Bigg[
\frac{1}{n}\sum_{\ell=1}^n\left(\nabla^2f(x_k+t_\ell s_k,\xi_\ell) - \nabla^2f(x_k+ts_k,\xi_\ell)\right) \\
+ \frac{1}{n}\sum_{\ell=1}^n\left(\nabla^2f(x_k+ts_k,\xi_\ell) - \nabla^2f(x_k+ts_k)\right)
\Big]s_k,
\end{multline*}
where we have added and subtracted $\nabla^2f(x_k+ts_k,\xi_\ell)$ terms.
Let $g_k^\prime$ be the sample average gradient defined on this $\{\xi_\ell\}$  minibatch.
Note that $g_k^\prime$ is not computed (since this minibatch is used for computing $\credr_k$) but its properties are the same as the properties of $g_k$ computed when building the model. 
Per our notation, $G_k^\prime$ and $N$ is the interpretation of $g_k^\prime$ and $n$, respectively, as random variables formulated in terms of all of the stochasticity in the algorithm.
From the expression for $\aredr_k - \credr_k$, we have
\begin{multline*}
{\mathbb E}\left[|\ared_k-\cred_k|\Big\vert \filt_{k-1/2}\right]
\leq {\mathbb E}\left[\|G_k^\prime-\nabla f(X_k)\|\Big\vert\filt_{k-1/2}\right]\|S_k\| \\
+\frac{1}{2}\left[\tilde L \|S_k\|
+ \frac{1}{N}\mathbb{E}\Big[\big\|\sum_{\ell=1}^N\nabla^2f(x_k+ts_k,\xi_\ell) - N\nabla^2f(x_k+ts_k)\big\|\Big\vert\filt_{k-1/2}\Big]\right]\|S_k\|^2.
\end{multline*}
Using the Cauchy-Schwarz inequality for conditional expectations,
\begin{multline}\label{eq:csce}
{\mathbb E}\left[|\ared_k-\cred_k|\Big\vert \filt_{k-1/2}\right] \\
\leq
\frac{\sigma_g\Delta_k}{\sqrt{N}} + \frac{1}{2}\widetilde{L}\Delta_k^3 + \frac{1}{2\sqrt{N}}{\rm var}\left(\nabla^2f(X_k+tS_k,\xi)\Big\vert\filt_{k-1/2}\right)^{1/2}\Delta_k^2.
\end{multline}
This bound shows that if $N\ge \Delta_k^{-2}$, then Assumption~\ref{a:new} is satisfied up to an order $\Delta_k^3$ term with $c_1 =\sigma_g$. 
Such a relationship between the minibatch size and the trust region radius is the same as the one for the gradient approximation.
We show that it implies Assumption~\ref{a:c} is satisfied as well.
Without loss of generality, we consider the case $\|h_k\|\ge\eta_2\delta_k$ since otherwise $\omega\in\mathcal{J}_k$ because $\credr_k=\aredr_k=0$.
By the fraction of Cauchy decrease condition (S2),
\begin{align*}
  \predr_k \ge \kfcd\|h_k\|\min\left\{\frac{\|h_k\|}{\kbmh},\delta_k\right\}
\ge \kappa_{\rm fcd}\eta_2\delta_k^2\min\left\{\frac{\eta_2}{\kappa_{\rm bmh}},1\right\}.
\end{align*}
Applying Markov's inequality with $N\ge \kappa\Delta_k^{-2}$ for some constant $\kappa$, we have from \eqref{eq:csce} that
\begin{align*}
\mathbb{P}\left(|\ared_k-\cred_k|\ge\eta\pred_k\Big\vert\filt_{k-1/2}\right)
&\le \frac{1}{\eta\pred_k}\mathbb{E}\left[|\ared_k-\cred_k|\Big\vert\filt_{k-1/2}\right] \\
&\le \frac{\sigma_g/\kappa}{\eta\kappa_{\rm fcd}\eta_2\min\left\{\frac{\eta_2}{\kappa_{\rm bmh}},1\right\}} + \mathcal{O}(\Delta_k).
\end{align*}
For $\kappa$ sufficiently large, this bound is smaller than $1-\beta$ to leading order.

In summary, we have shown that all of our oracle assumptions are satisfied by using minibatches of size $\mathcal{O}(\delta_k^{-2})$ under the \emph{common random numbers} setting from the simulation optimization literature.
This setting usually describes stochastic optimization under \eqref{eq:expected_grad}, which includes empirical risk minimization.
The $\mathcal{O}(\delta_k^{-2})$ sample complexity is an improvement over $\mathcal{O}(\delta_k^{-4})$ demonstrated for STORM but only because common random numbers were not considered in \cite{storm}.

If the stochastic process $\Delta_k$ obeys the same dynamics as in ProxSTORM (and in STORM) until the stopping time $T_\epsilon$, then $\Delta_k \geq {\cal O}(\epsilon)$ with high probability; see \cite{Jin25sample}.
The general bounds for the total complexity are given in \cite{Jin25sample} as $\mathcal{O}(\epsilon^{-4})$ in the common random number setting and $\mathcal{O}(\epsilon^{-6})$ in the more general setting of stochastic function estimates that are unbiased with bounded variance.
These bounds match the best-known stochastic optimization bounds in each setting.

\section{Numerical Results}\label{sec:ex}

We apply ProxSTORM to two examples: training a neural network classifier with $\ell^1$ regularization and a stochastic topology optimization problem.

\subsection{Training}

We use the HIGGS dataset for our training example.
The problem is ``to distinguish between a signal process [that] produces Higgs bosons and a background process [that] does not'' \cite{higgs}.
There are 11 million samples, $(x_k,y_k)$, in the dataset.
Each $x_k^\top\in\R^{28}$ represents a process and the corresponding $y_k$ labels whether Higgs bosons are produced by that process or not (one and zero, respectively).
Our model is a two-layer neural network:
\begin{equation}\label{eq:nn}
  m(x,\theta) := \sigma_2(\sigma_1(x_iW_1 + 1b_1^\top)W_2 + 1b_2^\top).
\end{equation}
The weights of the neural network are $\theta := (W_1,b_1,W_2,b_2)$.
The four components of $\theta$ belong to $\R^{28\times h},\R^{h\times 1}, \R^{h\times 1},$ and $\R$, respectively, where $h=10$ is the number of hidden units.
The ``1''s in \eqref{eq:nn} are vectors of ones with the appropriate dimension.
We take the activations $\sigma_1$ and $\sigma_2$ to be the logistic (i.e., sigmoid) function,
$$\sigma(x) = (1 + \exp(-x))^{-1},$$
applied to each component of the input.
We train \eqref{eq:nn} by minimizing its $\ell^1$-regularized classification error under the binary cross entropy loss
$$\ell(p,q) = q\log p + (1-q)\log(1-p).$$
We make the standard assumption that the data $\{(x_k,y_k)\}_k$ are independent and identically-distributed samples of random variables $(X,Y)$ having joint distribution $J$ \cite{hastie.2009}.
The optimization problem is
\begin{equation}\label{eq:bc}
\underset{\theta}{\text{minimize}}\;\; \mathbb{E}_J[\ell(m(X,\theta),Y)] + \lambda\|\theta\|_{\ell^1} \quad\text{with}\quad \lambda = 10^{-2}.
\end{equation}

To approximate the expectation, we view our data as a pool from which we sample.
Following the state of practice for this kind of training problems, we draw a fixed number of samples $n$ each iteration.
We take $n = 100$ and use those samples for the model gradient and Hessian, as well as computed reduction of the objective function for the iteration.
These details deviate from the analysis of ProxSTORM in the following ways:
\begin{enumerate}[i)]
\item The static value of $n$ is in some sense a heuristic choice of $\alpha$ and $\beta$.
\item Assumptions~\ref{a:c} and~\ref{a:new} involve conditioning on $\filt_{k-1/2}$ (i.e., the stochasticity in the models) which is not the case when using the model samples for the computed reduction.
\end{enumerate}
When computing a step that satisfies assumptions (S1) and (S2), we use a limited amount of Hessian information by employing the spectral proximal gradient method \cite[Algorithm 5]{baraldi.2022} for a maximum of two iterations.
Our Hessian applications are matrix free: with the $n$ samples selected for the iteration, we evaluate the application of the Hessian to vectors  instead of instantiating the Hessian as a matrix of values. 
This can be done in closed form, as we do, or using automatic differentiation.
Table~\ref{tbl:nn} documents our ProxSTORM parameters.
In our implementation, we used a maximum number of iterations $\mathfrak{m} = 50$ as a stopping condition.

\begin{table}[h!]
\caption{ProxSTORM parameters for our training example.}\label{tbl:nn}
\begin{tabular}{@{}p{\textwidth}@{}}
\centering
\begin{tabular}{l|l}
$\eta_1$           & $5\times 10^{-1}$ \\
$\eta_2$           & $5\times 10^{-5}$ \\
$\delta_0$         & $1\times 10^{1\hspace{0.2em}}$ \\
$\delta_{\rm max}$ & $1\times 10^{10\hspace{0.2em}}$ \\
$\gamma$           & $5\times 10^{0\phantom{-}}$
\end{tabular}
\end{tabular}
\end{table}

We ran 100 realizations of ProxSTORM.
At the end of each realization, we estimated the accuracy with which ProxSTORM solved \eqref{eq:bc}, i.e., the test error, by evaluating the objective function on the final value of $\theta$ using $10$,$000 = 100n$ data samples.
The train and test data samples were all independent and identically distributed draws from the same large data set, and the random seeds used to collect train and test samples differed for each ProxSTORM realization.
The $+45^\circ$-hatched part of Figure~\ref{fig:nnr} is a histogram of the ProxSTORM test errors.
The $-45^\circ$-hatched part is a histogram of the test error when solving the same problem \eqref{eq:bc} with Adam \cite{adam}, using the same seeds to again generate the same $n = 100$ samples per iteration and test sets.
Applying Adam in this way requires a ``gradient'' of $\|\cdot\|_{\ell^1}$, which we took to be the zero subgradient at zero and the gradient otherwise.
This convention is consistent with packages like PyTorch \cite{pytorch} and JAX \cite{jax}.
To compensate for the $\le 2$ ProxSTORM Hessian applications per iteration, we took $\mathfrak{m} = 150$ as the maximum number of Adam iterations.
\begin{figure}[h!]
\centering
\includegraphics[width=0.5\textwidth]{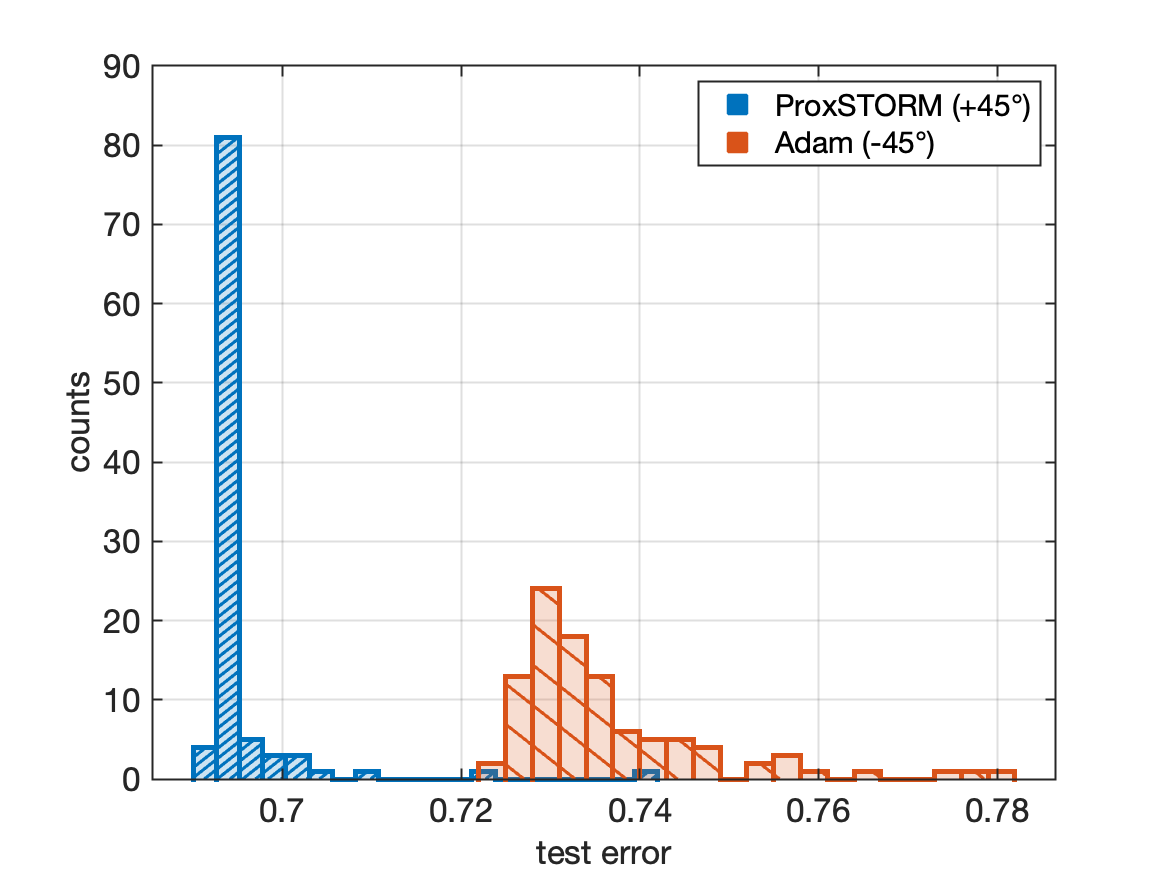}
\caption{Test error for 100 realizations of ProxSTORM (densely hatched) and Adam (sparsely hatched) applied to the $\ell^1$-regularized training example \eqref{eq:bc} for the HIGGS dataset \cite{higgs}.}\label{fig:nnr}
\end{figure}

The value of $\lambda$ is fixed in Figure~\ref{fig:nnr}, but we find that as the regularization parameter $\lambda\to 0$, Adam outperforms ProxSTORM in the test error metric.
When the nonsmoothness is non-negligible, however, ProxSTORM is more accurate and repeatable, as shown by Figure~\ref{fig:nnr}.
An additional comment in support of ProxSTORM is that the ``work complexity'' handicap we implemented through a larger $\mathfrak{m}$ for Adam is not always fair.
In some applications, the bottleneck is the sampling instead of how intensively the algorithm manipulates the data.
An example is when the model is not the neural network \eqref{eq:nn} but a linear system like a discretized partial differential equation whose operator depends on the stochastic variable.
Evaluating sample-average approximations of the objective function will require solves of these linear systems, which can be the most computationally intensive part of the problem.
Once computed, a factorization of the linear operator can be cached and reused to cheaply evaluate derivatives \cite{heinkenschloss.2008}.

\subsection{Topology Optimization}

Our next numerical example is a topology optimization problem that has convex constraints.
Let $D=(0,1)\times(0,2)$ be a two-dimensional domain for designing an elastic structure.
The left edge of the structure, $\Gamma_d := \{0\}\times(0,2)$, is fixed, and a unit force $T$ with uniformly distributed stochastic angle
$$-\frac{\pi}{6}\le \xi\le \frac{\pi}{6}$$
is applied to the structure at the middle of its right edge, as depicted in the left pane of Figure~\ref{fig:topopt}.
The goal is to select a material density
$$\rho: D\to [0,1]$$
for the structure that minimizes its expected compliance among all designs weighing a fraction of $\nu = 0.1$ of a fully-filled domain $D$ (corresponding to $\rho\equiv 1$).
The Young's modulus and Poisson ratio of the material is 200 gigapascals and 29 megapounds per square inch, respectively, corresponding to the properties of, e.g., A36 steel.
The problem is
\begin{equation}\label{eq:topopt1}
\begin{aligned}
&\underset{\rho\in L^2(D)}{\text{minimize}}\quad \int_{\partial D\setminus \Gamma_d} \mathbb{E}_\xi\Big[ \langle T(\xi), u(\rho,\xi)\rangle \big\vert_s\Big]\,\mathrm{d}s \\
&\textup{subject to}\;\; \int_D \rho\,\mathrm{d}s = v\vert D\vert, \quad 0\le\rho\le 1 \quad\text{almost everywhere}
\end{aligned}
\end{equation}
with $|D| = 2$ being the nondimensionalized weight of a fully-filled $D$ and $u(\rho,\xi)=u\in H^1(D)^2$ solves the weak form of the linear elasticity equations
\begin{align*}
  -\nabla\cdot(K(\rho):\varepsilon) &=0 &&\text{in $D$}, \\
  \varepsilon &= \tfrac{1}{2}(\nabla u + \nabla u^\top) &&\text{in $D$}, \\
  K(\rho):\varepsilon n &= T &&\text{on } \partial D\setminus \Gamma_d,\\
  u &=0 &&\text{on } \Gamma_d.
\end{align*}
In these equations, $n$ is the outward pointing normal vector and
\begin{equation}\label{eq:k}
  K(\rho) \coloneqq [\kappa_{\min} + (1-\kappa_{\min})\mathbb{F}(\rho)^3] K_0.
\end{equation}
The quantity $K_0$ is the usual isotropic stiffness tensor and $\mathbb{F}$ is the Helmholtz filter \cite{lazarov.2011}.
We set the parameter $\kappa_{\min} = 4\times 10^{-2}$.
To discretize \eqref{eq:topopt1}, we use a $100\times 200$ uniform quadrilateral mesh on which we use piecewise linear finite elements for $u$ and piecewise constant finite elements for $\rho$.
The discretized problem has $d=100\times 200 =$ 20,000 optimization variables, which we denote by $x$.
We have
\begin{align}\label{eq:topopt2}
\begin{split}
  &\underset{x\in\R^d}{\text{minimize}}\quad \mathbb{E}_\xi\left[\sum_{\ell=1}^2\sum_{k=1}^{2d} \tilde{T}(\xi)_{k\ell}\big[\tilde{K}^{-1}(x)\tilde{T}(\xi)\big]_{k\ell}\right] \\
  &\text{subject to}\;\; \sum_{m=1}^d w_m x_m = \nu|D| = 0.2, \quad
  0\le x_m\le 1\quad m = 1, 2, \ldots, d.
\end{split}
\end{align}
The quantity $\tilde{K}$ is a matrix that depends nonlinearly on $x$, just as $K$ depends nonlinearly on $\rho$ in \eqref{eq:k}.
The quantity $\tilde{T}(\xi)$ is a load vector corresponding to the traction force $T=T(\xi)$.
We obtain $\tilde{T}$ from a $\partial D\setminus \Gamma_d$ integral of $T$ against the piecewise linear test functions that discretize $u$.
In the linear constraint, the $\{w_m\}$ are integration weights for which
$$\sum_{m=1}^d w_mx_m \approx \int_D \rho\mathrm{d}s.$$
For more details about this example, see \cite{andreassen.2011,baiges.2019}.
\begin{figure}[h!]
\centering
\begin{tikzpicture}[every node/.style={draw,outer sep=0pt,thick}]
  \tikzstyle{ground}=[fill,pattern=north east lines,draw=none,minimum width=0.75cm,minimum height=0.3cm]
  \node (M) [fill=white,minimum width=3cm, minimum height=4.8cm] {};
  \node[draw=none] at (2cm,2.2cm) {$D$}; 
  \node (wall) [ground, rotate=-90, minimum width=4.8cm,yshift=-1.7cm] {};
  \draw (wall.north east) -- (wall.north west);
  \draw [-latex,thick,blue] (M.east) ++ (0cm,0) -- +(-0.8cm, 0.25cm);
  \draw [-latex,thick,blue] (M.east) ++ (0cm,0) -- +(-0.8cm,-0.25cm);
  \draw[thick, black] (M.east) -- ++( 1.3cm,0);
  \draw[thick, black] (M.east) -- ++(-0.8cm,0);
  \draw [dashed,blue] (M.east) ++ (0cm,0) --       +( 1.6cm,-0.50cm);
  \draw [dashed,blue] (M.east) ++ (0cm,0) --       +( 1.6cm, 0.50cm);
  \draw[blue] (M.east) ++(0.8cm,0) arc[start angle=0, end angle=-16, radius=0.8];
  \draw[blue] (M.east) ++(0.8cm,0) arc[start angle=0, end angle=+16, radius=0.8];
  \draw[red] (M.east) ++(1.1cm,0) arc[start angle=0, end angle=-16, radius=1.1];

  \node[draw=none,red,font=\normalsize] at (3cm,-0.2cm) {$\xi$};
\end{tikzpicture}\hspace{2em}
\includegraphics[height=0.25\textheight]{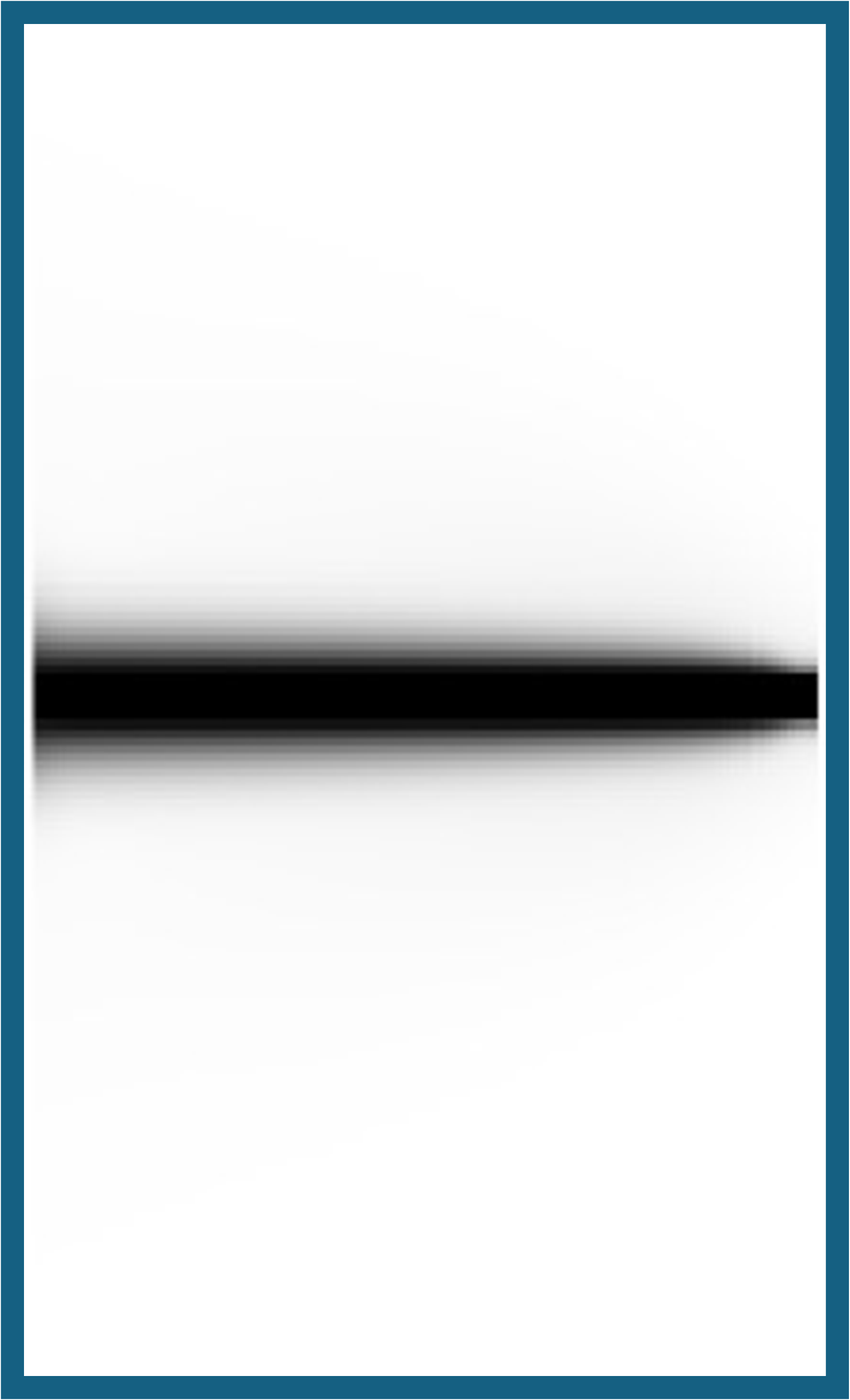}\hspace{1em}
\includegraphics[height=0.25\textheight]{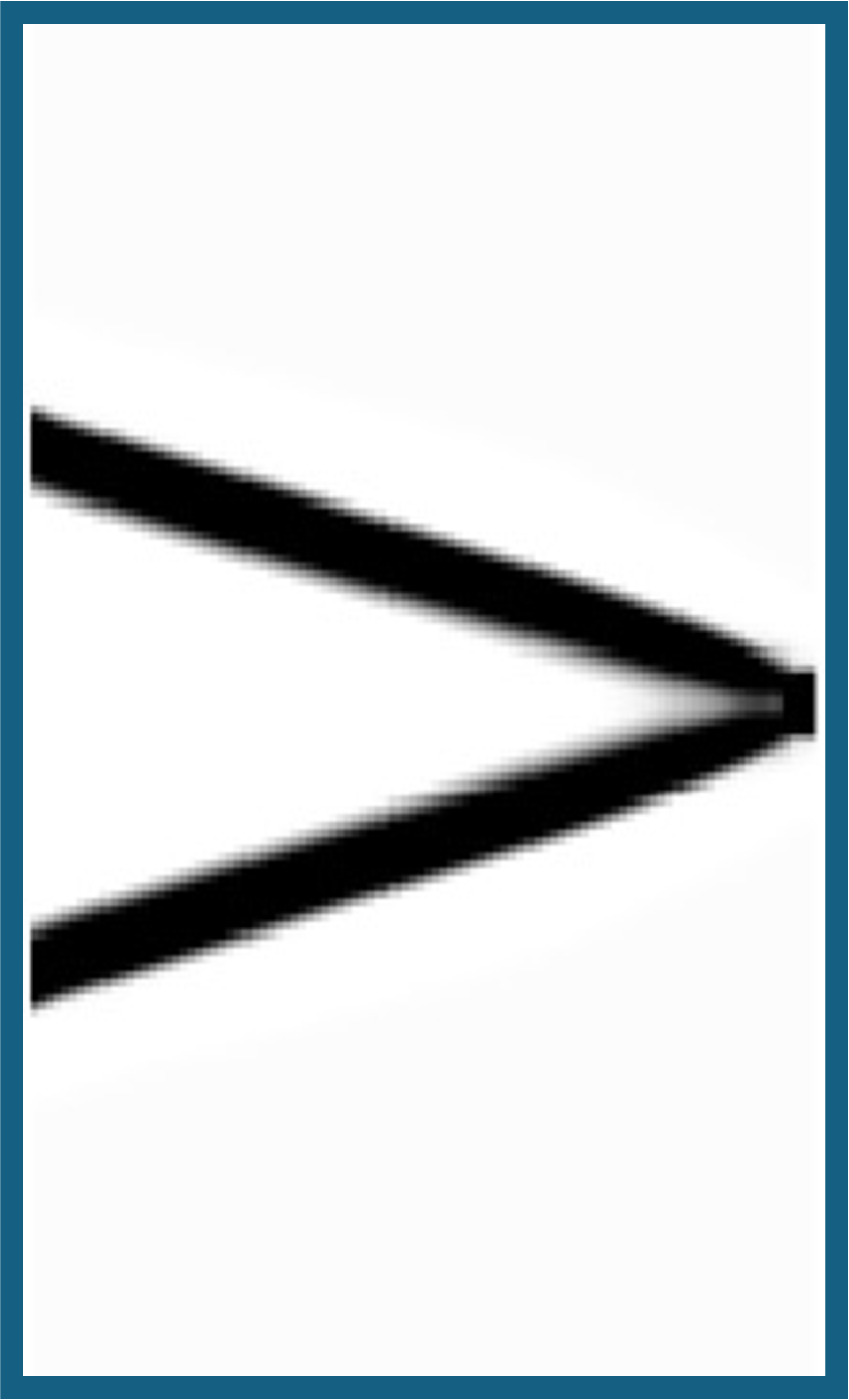}

\caption{Schematic of our topology example (left).
The optimal design when approximating the $\xi$ by its mean (middle) is qualitatively different than the optimal design under uncertainty (right).
}\label{fig:topopt}
\end{figure}

The problem \eqref{eq:topopt2} satisfies Assumption~\ref{a:p}
as it is a smooth objective with convex constraints;
hence, the indicator function (i.e., the function that is zero where the constraints are satisfied and infinity otherwise) is a valid choice of $\nobj$.

We solve \eqref{eq:topopt2} with ProxSTORM using the same parameter values in Table~\ref{tbl:nn}.
Figure~\ref{fig:topopt-test} is the analogue of Figure~\ref{fig:nnr}.
The problem \eqref{eq:topopt2} does not belong to the class of problems Adam is designed to solve, so instead we compare against the deterministic algorithm \cite[Algorithm 1]{baraldi.2022}. 
Both methods deal with the constraints via the proximal operator, as it can be computed exactly.
For the former method, samples of $\xi$ were redrawn each iteration, while for the latter method, the samples were held fixed across iterations, resulting in a sample average approximation solve \cite{shapiro.2021}.
In both settings, we used 3 independent and identically distributed samples of $\xi$ per iteration and ran the methods for a maximum of $\mathfrak{m} = 20$ iterations.
When the norm of the proximal gradient dropped below $10^{-10}$, we considered the deterministic method to be converged and terminated, resulting in the possibility of fewer than $\mathfrak{m}$ iterations.
For the deterministic method, we also used the more flexible trust-region update criteria in \cite[Algorithm 1]{baraldi.2022} with
$$(\eta_1,\,\eta_2) = (10^{-4},\,0.75) \quad\text{and}\quad (\gamma_1,\,\gamma_2) = (0.25,\,10).$$
Aside from the differences in sampling and trust-region updates (which ProxSTORM can be generalized to support),
the two algorithms were the same: they both used the same spectral proximal gradient subproblem solver as the training example but with the maximum number of iterations increased from 2 to 15.

\begin{figure}[h!]
\centering
\includegraphics[width=0.5\textwidth]{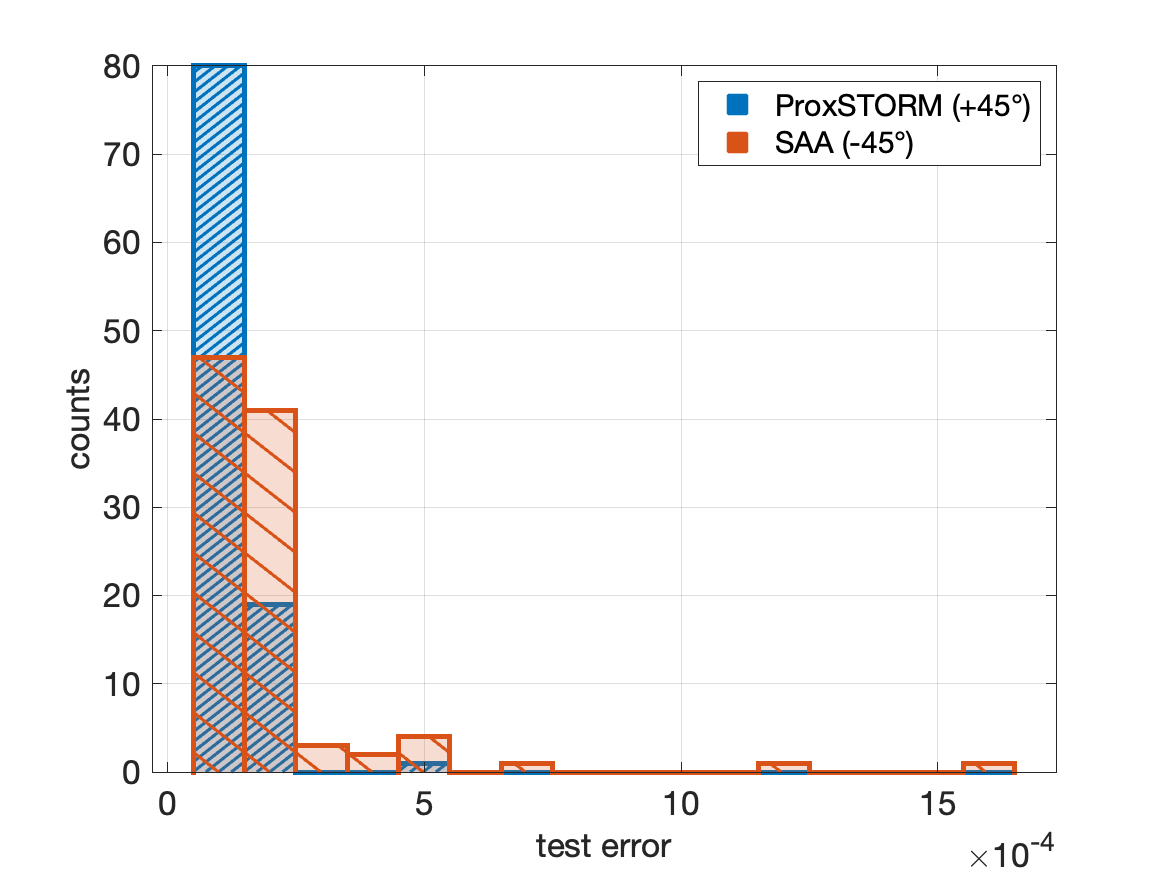}
\caption{Test error for 100 realizations of ProxSTORM (densely hatched) and the deterministic algorithm \cite{baraldi.2022} (sparsely hatched) applied to the topology example \eqref{eq:topopt2}.}\label{fig:topopt-test}
\end{figure}

Figure~\ref{fig:topopt-test} ksupports the notion that a fixed computational budget can go further in solving a stochastic optimization problem when using more samples, i.e., when seeing more of the problem.
This notion is the key distinction between ProxSTORM and its deterministic counterparts: the ProxSTORM analysis shows that ProxSTORM eventually solves the stochastic problem, while the solve of a sample average approximation never will.
For completeness, we include a comparison of the trust region radii of the two algorithms as Figure~\ref{fig:topopt-trr}.
\begin{figure}[h!]
\centering
\includegraphics[width=0.5\textwidth]{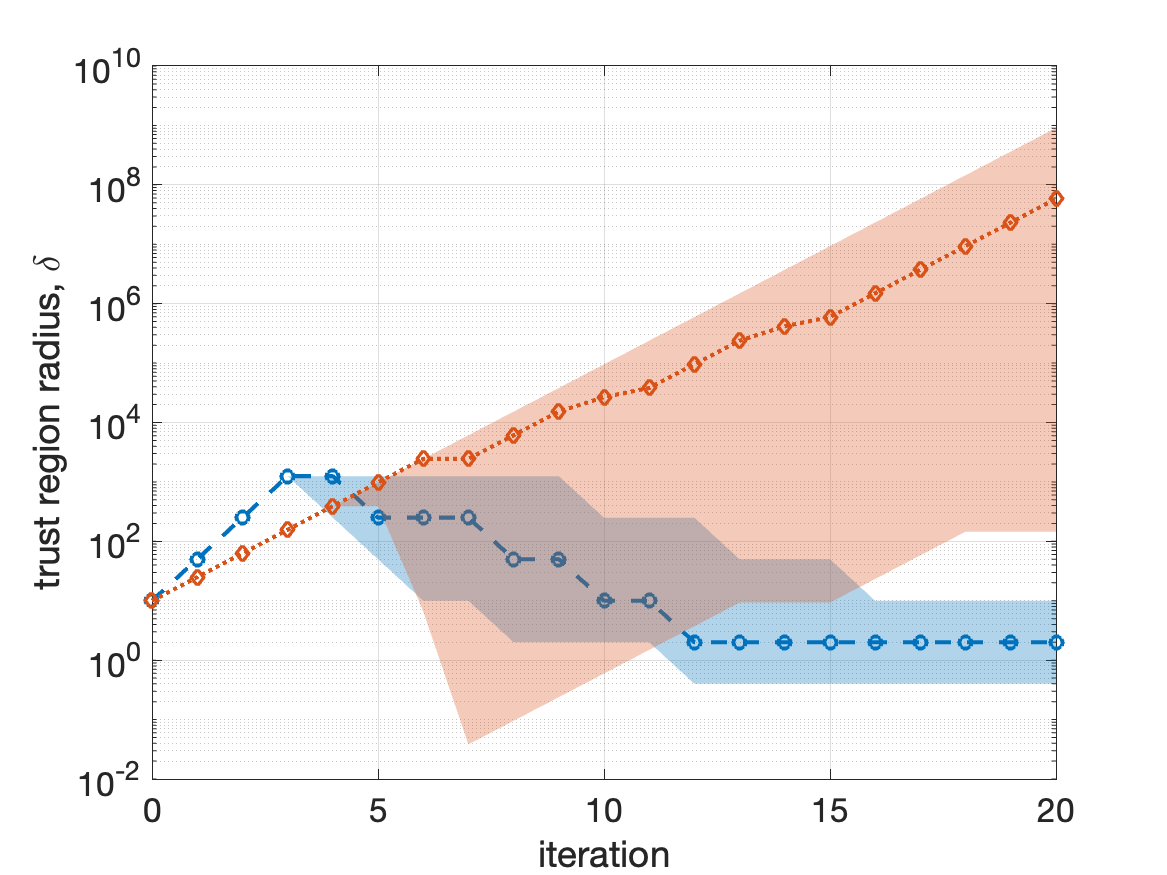}
\caption{The median trust region radii of ProxSTORM (circles connected with a dashed line) and the deterministic algorithm \cite{baraldi.2022} (diamonds connected with a dotted line).
In both cases, we include shaded regions that contain the largest and smallest trust region radii across the 100 trials.}\label{fig:topopt-trr}
\end{figure}

\section{Conclusion}\label{sec:cn}

We developed a new method, ProxSTORM (Algorithm~\ref{alg:ps}) for
objective functions that include convex but nonsmooth $\nobj$.
When $\nobj\equiv 0$, ProxSTORM reduces to STORM.
Our more general ProxSTORM algorithm inherits the key properties of first-order STORM; specifically, a limit-type global convergence guarantee (Theorem~\ref{thm:lim}) and an $\epsilon^{-2}$ complexity bound (Theorem~\ref{thm:co}).

Unlike ProxSTORM, deterministic trust-region methods \cite{conn.2000,baraldi.2022} do not typically require the trust-region radii to decrease to zero.
Close to an optimal solution, deterministic algorithms should take unencumbered Newton-like steps for fast convergence.
In contrast, the trust-region radii of the stochastic ProxSTORM and STORM methods satisfy the Robbins-and-Monro-like conditions \cite{robbins.1951}
$$
  \sum_{k=1}^\infty \Delta_k^2 < \infty \quad\text{and possibly}\quad \sum_{k=1}^\infty \Delta_k = \infty.
$$
In particular, $\lim_{k\to\infty} \Delta_k = 0$ (Corollary~\ref{cor:tozero}), meaning these stochastic trust-region methods will reject steps infinitely often, as that is the only mechanism for reducing $\Delta_k$.
For deterministic settings, inexactness conditions exist that do not drive the trust-region radii to zero.
The inexactness condition from \cite{baraldi.2022} for gradients of models
is effectively
$$
  \left\|\nabla M_k(X_k) - \nabla f (X_k)\right\| \le c\min\{\|H_k\|,\Delta_k\} \quad\text{for some constant } c.
$$
The stationarity condition appears on the right-hand side of this inequality, so the inequality is implemented using an iteration \cite[Algorithm 4]{baraldi.2022}.
Unfortunately, such iterations are ill-suited for stochastic settings like ours.
When an inequality does not hold with full probability, iterations like that for the condition above amplify the chances of an erroneous exit in a manner that is impractical to control.
Better understanding nuances like this in order to design rigorous yet practical stochastic algorithms is an interesting topic for future research.

\section{Acknowledgments}

This article has been authored by an employee of National Technology \& Engineering Solutions of Sandia, LLC under Contract No. DE-NA0003525 with the U.S. Department of Energy (DOE). The employee owns all right, title and interest in and to the article and is solely responsible for its contents. The United States Government retains and the publisher, by accepting the article for publication, acknowledges that the United States Government retains a non-exclusive, paid-up, irrevocable, world-wide license to publish or reproduce the published form of this article or allow others to do so, for United States Government purposes. The DOE will provide public access to these results of federally sponsored research in accordance with the DOE Public Access Plan:
$$\texttt{https://www.energy.gov/downloads/doe-public-access-plan}$$
This paper describes objective technical results and analysis. Any subjective views or opinions that might be expressed in the paper do not necessarily represent the views of the U.S. Department of Energy or the United States Government.

This work was supported by the Sandia Laboratory Directed Research and Development Program.

\begin{appendices}

\section{Proof of Theorem~\ref{thm:lim}}\label{sec:app1}

We first show that if the assumptions of Lemma~\ref{lem:incs} hold, then the ProxSTORM iterate and trust-region radius, $X_k$ and $\Delta_k$, respectively, satisfy
\begin{equation}\label{eq:sum2}
  \sum_{k=0}^{\infty}\mathbbm{1}\left(\|h(X_k)\|\ge\epsilon\right)\Delta_k <\infty
\end{equation}
with probability one for all positive real-valued random variables $\epsilon \in \filt$.
From the proof of Lemma~\ref{lem:incs}, we have that the right-hand side of \eqref{eq:bk2} bounds $\Circled{1}$ and that $\Circled{2}\le 0$.
It follows that there exist positive constants $c_7$ and $c_8$ such that
$$\mathbb{E}\left[(\dpsi)\big\vert\filt_{k-1}\right] - c_8\Delta_k^2 \le -\mathbbm{1}(\mathcal{B}_k)c_7\|h(X_k)\|\Delta_k.$$
Recall that $\mathbb{E}[\Psi_k - \Psi_{k+1}\vert\filt_{k-1}]$ is summable with probability one and so $\Delta_k$ is as well (Corollary~\ref{cor:tozero}).
As a result,
\begin{equation}\label{eq:limsum}
  \sum_{k=0}^{\infty}\mathbbm{1}\left(\|h(X_k)\|\ge\zeta\Delta_k\right)\|h(X_k)\|\Delta_k < \infty.
\end{equation}
with probability one.
Next, consider $\omega$ belonging to
\begin{equation}\label{eq:int2}
   \{({\rm\ref{eq:limsum}})\}
  \cap \{\epsilon > 0\}
  \cap \left\{ \lim_{k\to\infty}\Delta_k = 0\right\}.
\end{equation}
The event \eqref{eq:int2} is the intersection of probability one sets and thus has probability one.
Since $\omega$ belongs to the latter two events in \eqref{eq:int2}, there exists a $K(\omega)$ such that $\epsilon(\omega)\ge\zeta\Delta_k(\omega)$ for all $k\ge Kl\omega)$.
Hence, for all but finitely many $k$,
$$
   \mathbbm{1}(\|h(X_k)\|\ge\zeta\Delta_k)(\omega)\ge\mathbbm{1}(\|h(X_k)\|\ge\epsilon)(\omega).
$$
As a result,
\begin{equation*}
  \sum_{k=K(\omega)}^{\infty}\mathbbm{1}(\|h(X_k)\|\ge\epsilon)(\omega)\Delta_k(\omega)
  \le \sum_{k=K(\omega)}^{\infty}\mathbbm{1}(\|h(X_k)\|\ge\epsilon)(\omega)\frac{\|h(X_k(\omega))\|}{\epsilon(\omega)}\Delta_k(\omega)
\end{equation*}
with the right-hand side being finite since $\omega$ belongs to \eqref{eq:limsum}.
This argument shows that \eqref{eq:int2} is included in \eqref{eq:sum2}, so \eqref{eq:sum2} must have probability one.

We use \eqref{eq:sum2} holding with probability one to prove Theorem~\ref{thm:lim}.
The idea is to show that the probability of
\begin{equation}\label{eq:int3}
  \left\{\lim_{k\to\infty}\|h(X_k)\| \neq 0\right\} \cap \left\{\lim_{k\to\infty}\Delta_k = 0\right\}\cap\left\{\liminf_{k\to\infty}\|h(X_k)\| = 0\right\}
\end{equation}
is zero.
This will imply that the first set in the intersection has probability zero since the latter two sets have probability one.

Choose an $\omega$ in \eqref{eq:int3}.
Since $\omega$ belongs to the first set in \eqref{eq:int3}, there exists $\epsilon(\omega) > 0$ such that
$$\limsup_{k\to\infty} \|h(X_k(\omega))\| > 2\epsilon(\omega),$$
i.e., there exist infinitely many $k$ such that $\|h\left(X_k(\omega)\right)\| > 2\epsilon(\omega)$.
Since $\omega$ belongs to the third set in \eqref{eq:int3}, there exist infinitely many $k$ such that $\|h\left(X_k(\omega)\right)\| < \epsilon(\omega)$.
Thus, there are infinitely many pairs of indices $\left\{\big(k_{\ell}^{\prime}(\omega), k_{\ell}^{\prime\prime}(\omega)\big)\right\}_{\ell=1}^{\infty}$ for which
\begin{align*}
  \left\|h\left(X_{k_{\ell}^{\prime}(\omega)}(\omega)\right)\right\| < \epsilon(\omega) \quad\text{and}\quad
  \left\|h\left(X_{k_{\ell}^{\prime\prime}(\omega)}(\omega)\right)\right\| > 2\epsilon(\omega)
\end{align*}
with
\begin{align*}
  \|h\left(X_k(\omega)\right)\| \ge \epsilon(\omega) \quad\text{for all}\quad k_{\ell}^{\prime}(\omega) < k \le k_{\ell}^{\prime\prime}(\omega).
\end{align*}
We order these pairs of indices to be strictly increasing:
$
  k_1^{\prime}(\omega) < k_1^{\prime\prime}(\omega) < k_2^{\prime}(\omega) < \ldots
$
Observe that
\begin{multline*}
  \epsilon(\omega) < \left\|h\left(X_{k_{\ell}^{\prime\prime}(\omega)}(\omega)\right)\right\| -
                      \left\|h\left(X_{k_{\ell}^{\prime}(\omega)}(\omega)\right)\right\|
  \le \sum_{k=k_{\ell}^{\prime}(\omega)}^{k_{\ell}^{\prime\prime}(\omega)-1} \big\|h\left(X_{k+1}(\omega)\right) - h\left(X_k(\omega)\right)\big\|.
\end{multline*}
Since $\prox{r\nobj}(\cdot)$ is nonexpansive, the proximal gradient of $\sobj+\nobj$ is Lipschitz continuous.
In particular,
\begin{align*}
  \|h(y) - h(x)\|
&\le \frac{1}{r}\left(
 \left\|\prox{r\nobj}\left(y - r\nabla f(y)\right) - \prox{r\nobj}\left(x - r\nabla f(x)\right) \right\|
+\left\|y - x\right\|
\right) \\ 
  &\le \left(\frac{2}{r} + L\right)\|y - x\|.
\end{align*}
As a result,
\begin{align*}
  \epsilon(\omega) < \left(\frac{2}{r}+L\right)\sum_{k=k_{\ell}^{\prime}(\omega)}^{k_{\ell}^{\prime\prime}(\omega) - 1} \|X_{k+1}(\omega) - X_k(\omega)\|
  < \left(\frac{2}{r}+L\right)\sum_{k=k_{\ell}^{\prime}(\omega)}^{k_{\ell}^{\prime\prime}(\omega)}\Delta_k(\omega) .
\end{align*}
Since $\omega$ belongs to the second set in \eqref{eq:int3}, we have that for all $\ell$ sufficiently large,
$$\Delta_{k_{\ell}^{\prime}(\omega)}(\omega) < \frac{1}{2}\epsilon(\omega)\left(\frac{2}{r}+L\right)^{-1},$$
so for these $\ell$,
\begin{equation}\label{eq:unifbound}
  \frac{1}{2}\epsilon(\omega)\left(\frac{2}{r}+L\right)^{-1} < \sum_{k=k_{\ell}^{\prime}(\omega)+1}^{k_{\ell}^{\prime\prime}(\omega)}\Delta_k(\omega).
\end{equation}
We have that \eqref{eq:unifbound} holds for infinitely many $\ell$, meaning
$$\sum_{\ell=1}^{\infty}\sum_{k=k_{\ell}^{\prime}(\omega) + 1}^{k_{\ell}^{\prime\prime}(\omega)}\Delta_k(\omega) = \infty.$$
Note, however, that the definition of $\left\{\big(k_{\ell}^{\prime}(\omega),k_{\ell}^{\prime\prime}(\omega)\big)\right\}_{\ell=1}^{\infty}$
gives
\begin{equation*}
  \sum_{\ell=1}^{\infty}\sum_{k=k_{\ell}^{\prime}(\omega) + 1}^{k_{\ell}^{\prime\prime}(\omega)}\Delta_k(\omega)
  \le \sum_{k=0}^{\infty}\mathbbm{1}(\|h(X_k)\|\ge\epsilon)(\omega)\Delta_k(\omega),
\end{equation*}
so we have shown that \eqref{eq:int3} belongs to the complement of \eqref{eq:sum2}, which has probability zero.
We conclude \eqref{eq:int3} has probability zero, completing the proof.

\end{appendices}

\bibliographystyle{abbrvnat}
\bibliography{references.bib}

\end{document}